\documentclass[11pt,a4paper]{article}

\usepackage[utf8]{inputenc}
\usepackage[english]{babel}
\usepackage{amsfonts} 
\usepackage{amsmath}
\usepackage{amssymb}
\usepackage{amsthm}
\usepackage{mathrsfs}
\usepackage{microtype}
\usepackage{marvosym}
\usepackage{wasysym} 
\usepackage{stmaryrd}
\usepackage{dsfont}
\usepackage{enumerate}
\usepackage{multicol,multirow}
\usepackage{graphicx}
\usepackage{xcolor}
\usepackage{authblk}
\usepackage{url}

\usepackage{hyperref}
\hypersetup{
	colorlinks=false,
	linkcolor=blue,
	urlcolor=red,
	linktoc=all}

\usepackage{tikz}
\usetikzlibrary{decorations.pathreplacing}
\usetikzlibrary{decorations.pathmorphing}
\usetikzlibrary{calc}
\usetikzlibrary{decorations.pathreplacing}
\usetikzlibrary{decorations.pathmorphing}
\usetikzlibrary{arrows}
\usetikzlibrary{patterns,fadings}
\usepackage[all,cmtip]{xy}

\usepackage[tocflat]{tocstyle}
\usepackage{tocloft}

\newcommand{\Hilb}{\mathcal{H}}

\newcommand{\eps}{\text{$\varepsilon$}}
\newcommand{\oneop}{\mathds{1}}

\newcommand{\Xbar}{{\overline{X}}}

\newcommand{\rbar}{{\overline{r}}}
\newcommand{\iotabar}{{\overline{\iota}}}

\newcommand{\C}{\mathcal{C}}
\newcommand{\Z}{\mathcal{Z}}

\newcommand{\cS}{\mathcal{S}}

\newcommand{\A}{\mathcal{A}}

\newcommand{\cL}{\mathcal{L}}
\renewcommand{\L}{\cL}
\newcommand{\M}{\mathcal{M}}

\newcommand{\N}{\mathcal{N}}

\newcommand{\R}{\mathcal{R}}
\newcommand{\cR}{\mathcal{R}}
\renewcommand{\L}{\cL}

\newcommand{\RR}{\mathbb{R}}

\newcommand{\CC}{\mathbb{C}}

\newcommand{\NN}{\mathbb{N}}

\DeclareMathOperator{\Ind}{Ind}

\DeclareMathOperator{\Hom}{Hom}

\DeclareMathOperator{\Ad}{Ad}

\DeclareMathOperator{\tr}{tr}

\def\II{{I\!I}}

\newcommand{\lqq}{\lq\lq}

\usepackage{xspace}
\DeclareRobustCommand{\eg}{e.g.\@\xspace}
\DeclareRobustCommand{\cf}{cf.\@\xspace}
\DeclareRobustCommand{\ie}{i.e.\@\xspace}
\DeclareRobustCommand{\p}{p.\@\xspace}
\DeclareRobustCommand{\Sec}{Sec.\@\xspace}
\DeclareRobustCommand{\Prop}{Prop.\@\xspace}
\DeclareRobustCommand{\Lem}{Lem.\@\xspace}
\DeclareRobustCommand{\Cor}{Cor.\@\xspace}
\DeclareRobustCommand{\Thm}{Thm.\@\xspace}
\DeclareRobustCommand{\Ch}{Ch.\@\xspace}
\DeclareRobustCommand{\App}{App.\@\xspace}
\DeclareRobustCommand{\Ex}{Ex.\@\xspace}

\DeclareRobustCommand{\Def}{Def.\@\xspace}
\DeclareRobustCommand{\Rmk}{Rmk.\@\xspace}
\DeclareRobustCommand{\eq}{eq.\@\xspace}

\makeatletter
\DeclareRobustCommand{\etc}{%
    \@ifnextchar{.}%
        {etc}%
        {etc.\@\xspace}%
}
\makeatother

\newcommand{\Cstar}{$C^\ast$\@\xspace}

\def\u1net{{\A_\RR}}

\theoremstyle{plain}
\newtheorem{theorem}{Theorem}[section]
\newtheorem{corollary}[theorem]{Corollary}
\newtheorem{lemma}[theorem]{Lemma}
\newtheorem{proposition}[theorem]{Proposition}

\theoremstyle{definition}
\newtheorem{definition}[theorem]{Definition}

\theoremstyle{remark}
\newtheorem{example}[theorem]{Example}
\newtheorem{remark}[theorem]{Remark}

\numberwithin{equation}{section}

\begin{document}

\title{\huge Minimal index and dimension for 2-$C^*$-categories with finite-dimensional centers}

\author{\Large Luca Giorgetti}
\author{\Large Roberto Longo}
\affil{\normalsize Dipartimento di Matematica, Universit\`a di Roma Tor Vergata\\

Via della Ricerca Scientifica, 1, I-00133 Roma, Italy\\

{\tt giorgett@mat.uniroma2.it}, {\tt longo@mat.uniroma2.it}}

\date{}

\maketitle

\begin{abstract}
In the first part of this paper, we give a new look at inclusions of von Neumann algebras with finite-dimensional centers and finite Jones' index.
 The minimal conditional expectation is characterized by means of a canonical state on the relative commutant, that we call the spherical state; the minimal index is neither additive nor multiplicative (it is submultiplicative), contrary to the subfactor case.
 So we introduce a matrix dimension with the good functorial properties: it is always additive and multiplicative.
 The minimal index turns out to be the square of the norm of the matrix dimension, as was known in the multi-matrix inclusion case.
In the second part, we show how our results are valid in a purely 2-$C^*$-categorical context, in particular they can be formulated in the framework of Connes' bimodules over von Neumann algebras.
\end{abstract}

\tableofcontents

\vskip3cm

{\footnotesize Supported in part by the ERC Advanced Grant 669240 QUEST ``Quantum Algebraic Structures and Models'', MIUR FARE R16X5RB55W QUEST-NET, OPAL\,-``Consolidate the Foundations", GNAMPA-INdAM.}

\newpage
\section*{Introduction}
\addcontentsline{toc}{section}{Introduction}

In a landmark paper \cite{Jon83}, Jones initiated the study of inclusions of type $\II_1$ factors and defined an index for subfactors. If finite, the index is a positive number and, if less than 4, the possible values are bound to the set $\{4 \cos^2(\pi/k), k\in\NN, k\geq 3\} $, showing a remarkable quantization phenomenon. Moreover, every value in this set is the index of some subfactor, and it may also be realized by inclusions of multi-matrices, see also \cite{GdHJ89}.

The notion of index was generalized by Kosaki \cite{Kos86} to arbitrary subfactors $\N\subset\M$, or better to an arbitrary $E\in E(\M,\N)$, with $E(\M,\N)$ the set of normal faithful conditional expectations  $E:\M\rightarrow\N$. The index $\Ind(E)$ is again a positive number or infinity, its values are quantized as before, and gives back the Jones index in the type $\II_1$ subfactor case if $E$ is the trace-preserving conditional expectation. If the subfactor is not irreducible, namely if $\N'\cap\M \neq \CC\oneop$, there is more than one element in $E(\M,\N)$, yet there are expectations that minimize $\Ind(E)$ and one calls this minimal value the minimal index of $\N\subset\M$; such value is attained by exactly one expectation, the minimal one, see Hiai \cite{Hia88}, Longo \cite{Lon89} and Havet \cite{Hav90}.

The finiteness of the index is equivalent to the existence of a conjugate morphism $\iotabar: \M\rightarrow\N$ with respect to the inclusion morphism $\iota:\N\hookrightarrow\M$, where $\iota,\iotabar$ have to satisfy a conjugate equation  \cite{Lon90}. It was later shown \cite{LoRo97} how this leads to a notion of dimension in a purely tensor \Cstar-categorical (or better 2-\Cstar-categorical) setting, in the case of simple tensor unit. The solutions of the conjugate equations (or adjoint equations in the language of Mac Lane \cite{Mac98}) for $\iota$ and $\iotabar$ correspond to the conditional expectations $E\in E(\M,\N)$ and their dual expectations $E'\in E(\N',\M')$, both restricted to the common relative commutant $\N'\cap\M$ of the subfactors $\N\subset\M$ and $\M'\subset\N'$. Note that $\N'\cap\M = \Hom (\iota,\iota)$ and recall from \cite{Lon89}, \cite{Lon90} that $\iotabar(\M) \subset \N$ 
is isomorphic to the Jones basic extension $\M\subset\M_1$ determined by $E\in E(\M,\N)$, and $\M'\subset\N'$ is anti-isomorphic to $\M\subset\M_1$.

The aim of this paper is to define a notion of dimension and minimal index in the context of rigid 2-\Cstar-categories with finite-dimensional centers (not necessarily simple tensor units). We first treat the special case of inclusions of von Neumann algebras (not necessarily subfactors), in a language which is more familiar to operator-algebraists.

The first observation that one makes when dealing with non-factorial inclusions $\N\subset\M$ is that the index $\Ind(E)$, $E\in E(\M,\N)$, is not necessarily a scalar, but an element of $\Z(\M)$, the center of $\M$. Similarly $\Ind(E')$, $E'\in E(\N',\M')$, lies in $\Z(\N') = \Z(\N)$. Moreover, $E$ and $E'$ map $\N'\cap\M$ respectively onto $\Z(\N)$ and $\Z(\M)$, hence they need not be states of the relative commutant. Thus one has several options on how to compare the two restrictions $E_{\restriction \N'\cap\M}$ and $E'_{\restriction \N'\cap\M}$, \ie, several candidates for defining a sphericality property for solutions of the conjugate equations (comparison of the associated left and right inverses). 

The starting point of our work is determining the correct notion of sphericality for non-factorial inclusions (with finite index and finite-dimensional centers), namely the one that singles out the minimal expectations $E^0$ and ${E^0}'$, \ie, those minimizing respectively $\|\Ind(E)\|$ and $\|\Ind(E')\|$, among all other expectations (Theorem \ref{thm:minimalconnectedfindim}). Assuming without loss of generality that $\N\subset\M$ is connected, namely $\Z(\N) \cap \Z(\M) = \CC\oneop$, the sphericality property is controlled by two states on $\Z(\N)$ and $\Z(\M)$ respectively, that we call left and right state of the inclusion. Both the minimal index and these two canonical states are in turn determined by a further more fundamental invariant for the inclusion, a matrix, that we call the matrix dimension of the inclusion. The minimal index is the square of the $l^2$-norm of this matrix (hence we call the norm itself the scalar dimension of the inclusion), and the two previous states are determined by the unique left and right $l^2$-normalized Perron-Frobenius eigenvectors of the matrix dimension. These results on the minimal index turn out to be a natural generalization of those contained in \cite{Jon83}, \cite{GdHJ89} concerning the index for finite-dimensional \Cstar-algebras (multi-matrices), to other types of von Neumann algebras (not necessarily finite-dimensional, nor possessing a faithful normal trace). Moreover, in Theorem \ref{thm:minimalconnectedfindim} we give a weighted additivity formula that allows to compute the minimal index out of the matrix dimension and the left and right Perron-Frobenius eigenvectors of the inclusion. 

In Section \ref{sec:multi-mats}, we draw some consequences
on the easy observation that the matrix dimension of a multi-matrix inclusion is the inclusion matrix considered by Jones (Theorem \ref{thm:DisLambdaNinM}). 

In Section \ref{sec:s-extr}, we study the meeting-point of two aspects of index (the minimal one, and the one given by a trace) for multi-matrices (or more generally for finite von Neumann algebras). As in the factor case, we call an inclusion of finite von Neumann algebras (with finite index, finite-dimensional centers, and connected) extremal if the minimal expectation is the one which preservers the unique Markov trace of the inclusion (a trace on $\M$ that extends to the Jones tower generated by $\N\subset\M$). We call such an inclusion super-extremal if in addition the Markov trace itself, when restricted respectively to $\Z(\N)$ and $\Z(\M)$, coincides with the previously defined left and right states of the inclusion. In the case of multi-matrices, which are always extremal, we give a particularly simple characterization of super-extremality in terms of the inclusion matrix and of the dimensions of the full matrix algebras in the direct sum decompositions of $\N$ and $\M$, \ie, in terms of the Bratteli diagram of $\N\subset\M$ (Proposition \ref{prop:superextremalmultimat}). We show that the index of a super-extremal multi-matrix inclusion equals the ratio of the algebraic dimensions of $\M$ over $\N$, as it is the case for finite type $I$ subfactors, and that it must be an integer. Moreover every positive integer is realized as the index of such an inclusion. 

In Section \ref{sec:ifNorMfactor}, back to arbitrary inclusions with finite-dimensional centers, and assuming $\N$ or $\M$ to be a factor, we show that minimality of an expectation is equivalent to the commutation with its dual expectation on $\N'\cap\M$ (Proposition \ref{prop:minimaliffcommute}). 

In Section \ref{sec:multiplicativity}, we study the multiplicativity properties of the minimal index for non-factorial inclusions, namely its behaviour under composition of inclusions (this was the original motivation of our work). Contrary to the subfactor case \cite{KoLo92}, \cite{Lon92}, it turns out that the minimal index, or equivalently the scalar dimension, is not always multiplicative. The matrix dimension instead is always multiplicative (Theorem \ref{thm:dimensionmatmult}). We compare the minimal index of the composition of two inclusions with the product of the two minimal indices separately, in some particular cases. We give a general sufficient condition for multiplicativity (Proposition \ref{prop:minindexmultif}), namely the minimal index, or equivalently the scalar dimension, is multiplicative if the left state of the upper inclusion coincides with the right state of the lower inclusion, and in this case the minimal conditional expectation of the composed inclusion is the product of the two intermediate minimal expectations. Remarkably, the minimal index is always multiplicative for the particular subclass of multi-matrix inclusions considered in Section \ref{sec:s-extr}, namely the super-extremal ones (Proposition \ref{prop:multipsemm}). Furthermore, we study the minimal index in the non-factorial Jones tower of basic extensions, and show its multiplicativity in this case as well (Corollary \ref{cor:jonesextensionmultip}). 

In Section \ref{sec:additivity}, we study the additivity properties of the minimal index, which is again not additive in general. The matrix dimension instead is always additive (Proposition \ref{prop:dimensionmatadd}). In the case that $\N$ or $\M$ is a factor, the square of the scalar dimension, \ie, the minimal index itself, can be additive if we decompose either $\M$ or $\N$ along minimal central projections. 

In Section \ref{sec:2-cstar-cats}, which is the second main part of this work, we develop the theory of minimal index and dimension (as a matrix) for 2-\Cstar-categories with non-simple tensor units (the 1-arrow units for the horizontal composition). We push the analogy with the theory of minimal index for non-factorial inclusions developed in the previous sections, or better for Connes' bimodules between algebras with finite-dimensional centers described in \cite{Lon17arxiv}, ending up in a complete generalization of the results contained in \cite{LoRo97} concerning the simple unit case. The crucial step is the definition of standard solution of the conjugate equations (Definition \ref{def:stdsol}), obtained in the case of connected 1-arrows as a weighted direct sum of the standard solutions (in the sense of \cite{LoRo97}) associated with the corresponding factorial sub-1-arrows. The weights are given by the left and right Perron-Frobenius eigenvectors of the matrix dimension and they cannot be omitted in order to characterize, as we do, sphericality (Theorem \ref{thm:stdiffspherical}) and minimality (Theorem \ref{thm:stdiffminimal}) properties of solutions of the conjugate equations by means of standardness. This last section is the most novel part of the work, it is self-contained and it generalizes the results of Section \ref{sec:inclusions}, \ref{sec:multiplicativity} and \ref{sec:additivity}.

As an outlook, we intend to study further the theory of minimal index, and its possibly 2-$W^*$-categorical translation and generalization, in the most general case of von Neumann algebras with arbitrary centers, extending our analysis to the infinite-dimensional centers case.

\section{Minimal index, spherical state and matrix dimension}\label{sec:inclusions}

Let $\N\subset\M$ be a unital inclusion of $\sigma$-finite von Neumann algebras acting on a Hilbert space $\Hilb$. Throughout this paper we shall always deal with $\sigma$-finite von Neumann algebras and unital inclusions (\ie, $\N$ and $\M$ have the same identity operator $\oneop = \oneop_\Hilb$) without further mentioning. Denoted by $E(\M,\N)$ the set of all normal faithful conditional expectations from $\M$ onto $\N$, the inclusion $\N\subset\M$ is called \emph{finite-index} if $E(\M,\N)\neq\emptyset$ and there is an expectation $E\in E(\M,\N)$ with finite Jones index, denoted by $\Ind(E)$. See \cite{Kos86} and \cite{BDH88}, \cite{Pop95} for the definition of $\Ind(E)$, which is in general a positive invertible element of $\Z(\M) = \M'\cap\M$, the center of $\M$, if $E$ has finite index. In fact, $\Ind(E) \geq \oneop$.

\begin{definition}\label{def:minimalindex}
A conditional expectation $E\in E(\M,\N)$ with finite index is called \textbf{minimal} if the norm of its index $\|\Ind(E)\|$ is minimal among all other elements of $E(\M,\N)$. This minimal value is called the \textbf{minimal index} of the inclusion $\N\subset\M$.
\end{definition}

At this level of generality, it is a result of Jolissaint \cite[\Thm 1.8]{Jol91} that guarantees the \emph{existence} of minimal expectations for finite-index inclusions of von Neumann algebras.

\begin{definition}
The inclusion $\N\subset\M$ is called \textbf{connected} if $\Z(\N)\cap \Z(\M) = \CC\oneop$. This is obviously equivalent to $\M'\subset\N'$ being connected. 
\end{definition}

Now assume that $\N\subset\M$ is \emph{connected} and that $\N$ and $\M$ have \emph{finite-dimensional centers}. In this case, we know by \cite[\Thm 2.9 (ii)]{Hav90}, \cite[\Prop 3.1]{Ter92} that there exists a \emph{unique} minimal expectation $E^0\in E(\M,\N)$, \ie, $\|\Ind(E^0)\|$ is minimal among all other $\|\Ind(E)\|$, furthermore $\Ind(E^0)$ turns out to be a scalar multiple of the identity, \ie, $\Ind(E^0) = c \oneop$, where $c:=\|\Ind(E^0)\|\geq 1$ is the minimal index of $\N\subset\M$.

\begin{remark}
Recall that if there is an expectation $E\in E(\M,\N)$ with finite index, then the finite-dimensionality of $\Z(\N)$, $\Z(\M)$ or $\N'\cap\M$ are equivalent conditions, see \cite[\Cor 3.18, 3.19]{BDH88} for a proof of this fact. Furthermore, when $\N'\cap\M$ is finite-dimensional we know by \cite[\Thm 6.6]{Haa79II}, \cite[\Thm 5.3]{CoDe75} that either \emph{every} conditional expectation has finite index, \ie, the operator-valued weight $E^{-1}$ defined in \cite{Kos86} is bounded on $\N'$, or $\|\Ind(E)\| = \infty$ for every $E\in E(\M,\N)$.
\end{remark}

\begin{remark}
Connected inclusions are the building blocks in the analysis of the index of expectations between algebras with non-trivial centers, in the sense that an inclusion of algebras $\N\subset\M$, equipped with a conditional expectation $E\in E(\M,\N)$, can always be written as direct sum (or direct integral) of connected inclusions, also endowed with conditional expectations, by taking the decomposition along $\Z(\N)\cap\Z(\M)$, \cf \cite[\Rmk 3.4]{Ter92}, \cite[\Thm 1]{FiIs96}.
In the non-connected case (but still assuming finite-dimensional centers) there is a unique minimal expectation, see Definition \ref{def:minimalindex}, such that in addition it has minimal index (in the sense of operators in $\Z(\M)$) among all other minimal expectations, see \cite[\Thm 2.9 (iii)]{Hav90}, \cite[\Thm 3.5]{Iso95}. The index of this expectation is an element of $\Z(\N)\cap\Z(\M)$.
In the case of infinite-dimensional centers instead, even assuming them to be atomic, the connectedness of the inclusion does \emph{not} guarantee uniqueness of the minimal expectation, see \cite[\Prop 10, \Sec 5]{FiIs96}.
\end{remark}

In the following Theorem \ref{thm:minimalconnectedfindim} we give an intrinsic characterization of minimality for a normal faithful conditional expectation $E\in E(\M,\N)$, \ie, for its dual conditional expectation $E' := \Ind(E)^{-1}E^{-1}\in E(\N',\M')$, improving a result of Teruya, see \cite[\Thm 3.3]{Ter92}, by using Perron-Frobenius theory for non-negative matrices. We also characterize the minimal index, \cf \cite{Lon89}, \cite{Lon90}, \cite{LoRo97}, as the Perron-Frobenius eigenvalue of a certain \lqq matrix dimension", similarly to the index given by the trace in the case of finite von Neumann algebras \cite{Jon83}, \cite{GdHJ89}, \cite{Jol90}.

First we introduce some notation. Denote by $\{q_j, j=1,\dots,n\}$, $\{p_i, i=1,\ldots,m\}$ the minimal central projections of $\N$, $\M$, respectively, where $n=\dim(\Z(\N))$, $m=\dim(\Z(\M))$. In particular, $\sum_j q_j = \oneop$ and $\sum_i p_i = \oneop$, the $q_j$ are mutually orthogonal, similarly for the $p_i$, and clearly $p_iq_j = q_jp_i$. For every $E\in E(\M,\N)$, let $\lambda_{ij}\geq 0$, $\lambda'_{ji}\geq 0$ be such that 
$$\lambda_{ij} q_j = E(p_i)q_j,\quad \lambda'_{ji} p_i = E'(q_j)p_i,$$ 
then $\lambda_{ij}> 0$, or equivalently $\lambda'_{ji}> 0$, if and only if $p_iq_j \neq 0$. Whenever $p_iq_j \neq 0$, define $E_{ij}\in E(\M_{ij}, \N_{ij})$ by $E_{ij}(q_j x p_i q_j) := \lambda_{ij}^{-1} E(x p_i) p_i q_j$, $x\in\M$, and define $E'_{ji}\in E(\N'_{ji}, \M'_{ji})$ by $E'_{ji}(p_i y q_j p_i) := {\lambda'}_{ji}^{-1} E'(y q_j) q_j p_i$, $y\in\N'$, where $\M_{ij} := q_j \M p_i q_j$, $\N_{ij} := \N p_i q_j$ and $\N_{ij}\subset\M_{ij}$ is a subfactor. 
Let $\Lambda$ and $\Lambda'$ be respectively the $m\times n$ and $n\times m$ matrices with entries $\Lambda_{i,j} := (\lambda_{ij})_{i,j}$ and $\Lambda'_{j,i} := (\lambda'_{ji})_{j,i}$ and observe that they are column-stochastic, \ie, $\sum_i \lambda_{ij} = 1$ and $\sum_j \lambda'_{ji} = 1$. We refer to $\Lambda$ and $\Lambda'$ as the \lqq expectation matrices" respectively of $E$ and $E'$. 

Recall that by \cite[\Prop 2.2, 2.3]{Hav90}, \cite[\Prop 2.1]{Ter92} every expectation $E\in E(\M,\N)$ is uniquely determined by $\Lambda$ and $(E_{ij})_{i,j}$ via the formula
\begin{equation}\label{eq:decompositionofE}
E(x) = \sum_{i,j} \lambda_{ij} \sigma_{ij}(E_{ij}(q_j x p_i q_j))
\end{equation}
where $\sigma_{ij} : \N_{ij} \rightarrow \N_j := \N q_j$ the inverse of the induction isomorphism $y q_j \mapsto y p_i q_j$, $y\in\N$ ($p_i$ has central support $\oneop$ in $\N_j '$). Similarly, every $E'\in E(\N',\M')$ is determined by $\Lambda'$ and $(E'_{ij})_{i,j}$.

\begin{lemma}\label{lem:connectediffirred}
The inclusion $\N\subset\M$ is connected if and only if the $m\times n$ matrix $S$ with $(i,j)$-th entry equal to $1$ (or any positive number) if $p_iq_j \neq 0$ and equal to $0$ if $p_iq_j = 0$, is indecomposable, which is equivalent to $S S^t$, or equivalently $S^t S$, being irreducible square matrices. 
\end{lemma}

\begin{proof}
Suitable modification of the arguments in \cite[\Prop 2.3.1 (f)]{GdHJ89} and \cite[\Lem 1.3.2 (b)]{GdHJ89}.
\end{proof}

\begin{theorem}\label{thm:minimalconnectedfindim}
Let $\N\subset\M$ be a finite-index connected inclusion of von Neumann algebras with finite-dimensional centers, and denote by $\{q_j, j=1,\dots,n\}$, $\{p_i, i=1,\dots,m\}$ the minimal central projections as above. Let $E\in E(\M,\N)$, the following conditions are equivalent:
\begin{itemize}
\item[$(a)$] $E=E^0$, \ie, $E$ is the minimal conditional expectation in $E(\M,\N)$, or equivalently $E'$ is the minimal conditional expectation in $E(\N',\M')$.
\item[$(b)$] 
There is a normal faithful state $\omega_s$ on $\N'\cap\M$ such that $\omega_s = \omega_s \circ E = \omega_s \circ E'$ on $\N'\cap\M$, or equivalently
\begin{equation*}
\omega_l \circ E = \omega_r \circ E' \quad \text{on} \quad \N'\cap\M
\end{equation*}
where $\omega_l$ and $\omega_r$ are the two normal faithful states obtained by restricting $\omega_s$ to $\Z(\N)$ and $\Z(\M)$ respectively.

In this case, the states $\omega_l$, $\omega_r$ and $\omega_s$ are uniquely determined, and we refer to them respectively as \lqq left", \lqq right" and \lqq spherical" state of the inclusion. Moreover, $\omega_s$ is a trace on $\N'\cap\M$.
\item[$(c)$] For every $i=1,\ldots, m$, $j=1,\ldots, n$ denote by $c_{ij}\geq 1$ the minimal index of the subfactor $\N_{ij}\subset\M_{ij}$ if $p_iq_j \neq 0$, and set $c_{ij} := 0$ otherwise. Let $d_{ij} := c_{ij}^{1/2}$ be the intrinsic dimension of the respective inclusion morphism and $D$ the $m \times n$ \lqq matrix dimension" with entries $D_{i,j} := (d_{ij})_{i,j}$. Let $(\nu_j^{1/2})_j$ and $(\mu_i^{1/2})_i$ be the Perron-Frobenius eigenvectors (with positive entries, normalized such that $\sum_j \nu_j =1$ and $\sum_i \mu_i = 1$) respectively of $D^t D$ and $D D^t$, namely
\begin{equation}\label{eq:PF-eqns}
\begin{aligned}
D^t D \nu^{1/2} = d^2 \nu^{1/2}\\
D D^t \mu^{1/2} = d^2 \mu^{1/2}
\end{aligned}
\end{equation}
where $d^2>0$ denotes the common spectral radius (the Perron-Frobenius eigenvalue), \ie, $d=\|D\|$, the operator norm of $D$ between Hilbert spaces $\CC^n$ and $\CC^m$ with the $l^2$-norm.

Then $\|\Ind(E_{ij})\| = c_{ij}$  for every $i,j$, moreover $c_{ij} = c \lambda_{ij} \lambda'_{ji}$ and $\lambda_{ij} = \mu_i \nu_j^{-1} \lambda'_{ji}$ for some constant $c$.

In this case, the index of $E$ is a scalar multiple of the identity and $\|\Ind(E)\| = c$, \ie, $\Ind(E) = c \oneop$. Moreover, $c$ is the minimal index of the inclusion $\N\subset\M$, and it holds  
\begin{equation}\label{eq:cisdsquare}
c=d^2
\end{equation}
\ie, the minimal index equals the previous spectral radius (the squared norm of the matrix dimension). $E$ is uniquely determined by setting $E_{ij} = E^0_{ij}$, the unique minimal expectation in $E(\M_{ij}, \N_{ij})$, and
\begin{equation}\label{eq:DdeterminesLambda}
\lambda_{ij} = \frac{d_{ij}}{d} \frac{\mu_i^{1/2}}{\nu_j^{1/2}}
\end{equation}
for every $i,j$. Similarly for $E'$ after exchanging $E_{ij}$, $E^0_{ij}$ with $E'_{ji}$, ${E^0_{ji}}'$ respectively, $\lambda_{ij}$ with $\lambda'_{ji}$, 
and exchanging the roles of $\nu_j^{1/2}$ and $\mu_i^{1/2}$ in equation (\ref{eq:DdeterminesLambda}).
\end{itemize}

The \lqq left" and \lqq right" states $\omega_l$, $\omega_r$ determine respectively the eigenvectors $\nu^{1/2}$, $\mu^{1/2}$, and vice versa, via
\begin{equation*}
\nu_j = \omega_l(q_j),\quad \mu_i = \omega_r(p_i),
\end{equation*}
hence we call $\nu^{1/2}$ and $\mu^{1/2}$ respectively the \lqq left" and \lqq right" Perron-Frobenius eigenvector of the inclusion. The \lqq scalar dimension" $d$ of the inclusion $\N\subset\M$ and the $d_{ij}$ are related by
\begin{equation}\label{eq:weightedadditivd}
d = \sum_{i,j} d_{ij} \nu_j^{1/2} \mu_i^{1/2}.
\end{equation}
\end{theorem}

\begin{proof}
The equivalence of $(a)$ and $(b)$ is essentially contained in \cite{Ter92}, we just have to observe that 
$$E_{\restriction \N'\cap\M}:\N'\cap\M \rightarrow \Z(\N), \quad E'_{\restriction \N'\cap\M}: \N'\cap\M \rightarrow \Z(\M)$$ 
and that in the proof of \cite[\Thm 3.3]{Ter92} only the existence of an $E$-invariant and $E'$-invariant state on $\N'\cap\M$ is actually needed,
\footnote{We shall re-prove and generalize the equivalence of $(a)$ and $(b)$, \ie, of \lqq minimality" and \lqq sphericality" properties in the setting of rigid 2-\Cstar-categories, see Section \ref{sec:2-cstar-cats}, Theorem \ref{thm:stdiffspherical} and \ref{thm:stdiffminimal}.}.
By connectedness assumption, such a state $\omega_s$ is unique and explicitly given by the formula
\begin{equation}\label{eq:sphericalstate}
\omega_s(x)\oneop = P(x), \quad x\in\N'\cap\M
\end{equation}
where $P:\N'\cap\M \rightarrow \Z(\N)\cap\Z(\M) = \CC\oneop$ is the pointwise strong limit of the words $\{EE'EE'\ldots, E'EE'E\ldots\}$, hence also the pointwise norm limit given that $\N'\cap\M$ is finite-dimensional. The convergence can be checked by representing $E$ and $E'$ as projections on a common Hilbert space and applying von Neumann's ergodic theorem \cite[\Thm II.11]{RS1} to compute the projection on the intersection. 
The trace property of $\omega_s$ follows from the fact that $E=E^0$ (similarly for $E'={E^0}'$) is \emph{b\'ecarre} in the terminology of \cite[\Thm 2.9]{Hav90}, namely $E(xy) = E(yx)$ for every $x\in\N'\cap\M$ and $y\in\M$.

To show the equivalence of $(b)$ and $(c)$, we first assume the existence of an $E$-invariant and $E'$-invariant state $\omega_s$ on $\N'\cap\M$.
We let $\mu_i := \omega_s(p_i)$, $\nu_j := \omega_s(q_j)$ and $c$ be the minimal index of $\N\subset\M$. Thus $\sum_i \mu_i = 1$ and $\sum_j \nu_j = 1$. Moreover, $\mu_i = \sum_j \lambda_{ij} \nu_j$ for every $i$ and $\nu_j = \sum_i \lambda'_{ji} \mu_i$ for every $j$ using $E(p_i) = \sum_j \lambda_{ij} q_j$ and $E'(q_j) = \sum_i \lambda'_{ji} p_i$, together with the invariance properties of $\omega_s$. Similarly we get $\omega_s(p_i q_j) = \lambda_{ij} \nu_j = \lambda'_{ji} \mu_i$, hence $\lambda_{ij} = \mu_i \nu^{-1}_j \lambda'_{ji}$. Now, as in \cite[\Thm 3.3 $(c)\Rightarrow(b)$]{Ter92}, with suitable modifications of the notation, we can argue that  
$$\|\Ind(E_{ij})\| = c_{ij} = c \lambda^2_{ij} \frac{\nu_j}{\mu_i} = c \lambda_{ij} \lambda'_{ji}$$
where $c_{ij}$ is the minimal index of the subfactor $\N_{ij}\subset\M_{ij}$, \ie, $E_{ij}$ is the minimal expectation in $E(\M_{ij},\N_{ij})$ (or equivalently, $E'_{ji}$ is the minimal expectation in $E(\N'_{ji},\M'_{ji})$). The expression of $\lambda_{ij}$ (the entries of the expectation matrix associated with $E^0$) appearing in equation (\ref{eq:DdeterminesLambda}) follows by solving the two previous equations for $\lambda_{ij}$ and $\lambda'_{ji}$. The equivalence of $(b)$ and $(c)$ now follows as in \cite[\Thm 3.3]{Ter92} once we observe that the condition that $\Ind(E)$ is a scalar \cite[\Thm 3.3 $(b)$ $(ii)$]{Ter92} is a consequence of our assumptions $(c)$. Indeed we have
$$\Ind(E) = \sum_{i,j} \|\Ind(E_{ij})\|\lambda^{-1}_{ij} p_i = \sum_{i,j} c \lambda'_{ji} p_i = c\oneop$$
by the formula for the index of an arbitrary conditional expectation in $E(\M,\N)$ \cite[\Thm 2.5]{Hav90}, see also \cite[\Prop 2.3]{Ter92}.

The states $\omega_l$ and $\omega_r$ are clearly determined by their values on the minimal projections in the centers $\Z(\N)$ and $\Z(\M)$ respectively. The statement about the Perron-Frobenius theoretical characterization of our numerical data $\nu^{1/2}$, $\mu^{1/2}$ and $d$ (hence of the minimal index $c$ of $\N\subset\M$) can be seen as follows. The norm of the index of an arbitrary conditional expectation in $E(\M,\N)$ \cite[\Rmk 2.6]{Hav90}, \cite[\eq (2.5)]{Ter92} reads 
$$\|\Ind(E)\| = \max_i \{\sum_j \|\Ind(E_{ij})\| \lambda_{ij}^{-1}\}$$
hence for $E=E^0$, using $(c)$, we get
$$c = \sum_j c_{ij}\lambda_{ij}^{-1} = \sum_j c_{ij}^{1/2} c^{1/2} \nu_j^{1/2} \mu_i^{-1/2}$$
where the sum over $j$ is independent of $i$, \cf \cite[\Prop 2.6]{Ter92}, thus
$$c^{1/2} \mu_i^{1/2} = \sum_j d_{ij} \nu_j^{1/2}.$$
Similarly for ${E^0}'$, using $\|\Ind(E)\| = \|\Ind(E')\|$ and again $(c)$, we get
$$c^{1/2} \nu_j^{1/2} = \sum_i d_{ij} \mu_i^{1/2}.$$
In matrix notation, we obtained $D \nu^{1/2} = c^{1/2} \mu^{1/2}$ and $D^t \mu^{1/2} = c^{1/2} \nu^{1/2}$. Now the statements about $DD^t$, $D^tD$ and $d$, namely equations (\ref{eq:PF-eqns}), (\ref{eq:cisdsquare}), hence (\ref{eq:DdeterminesLambda}) and (\ref{eq:weightedadditivd}), easily follow by the Perron-Frobenius theory for irreducible non-negative matrices \footnote{In which case, the two systems of equations $D \nu^{1/2} = c^{1/2} \mu^{1/2}$, $D^t \mu^{1/2} = c^{1/2} \nu^{1/2}$ and $D^t D \nu^{1/2} = c \nu^{1/2}$, $D D^t \mu^{1/2} = c \mu^{1/2}$ are actually equivalent.}, see, \eg, \cite[\Thm 8.3.4, 8.4.4]{HoJoBook}, thanks to Lemma \ref{lem:connectediffirred}. In particular, the minimal index $c$ of the inclusion and the spectral radius $d^2$ of the matrix dimension coincide, thus the proof is complete.
\end{proof}

\begin{corollary}
In the previous assumptions, the minimal index $d^2 = \|D\|^2$ of the inclusion $\N\subset\M$ is either in the Jones' discrete series $\{4 \cos^2(\pi/k), k\in\NN, k\geq 3\}$ or in $[4,\infty[$.
\end{corollary}

\begin{proof}
Follows from Theorem \ref{thm:minimalconnectedfindim}, from the quantization of the minimal index for subfactors \cite[\Cor 4.3]{Lon89} and from \cite[\Thm 1.1.3, \Prop 3.7.12 (c)]{GdHJ89}.
\end{proof}

\begin{definition}\label{def:dimensionmatrix}
We call \textbf{matrix dimension}, or simply \textbf{dimension}, of the inclusion $\N\subset\M$ the $m\times n$ matrix $D$ defined in Theorem \ref{thm:minimalconnectedfindim}. Its $l^2$-operator norm $d=\|D\|$ will be referred to as the \textbf{scalar dimension} of $\N\subset\M$.
\end{definition}

\begin{remark}
Notice that by Theorem \ref{thm:minimalconnectedfindim} all the numerical invariants appearing in the theory of minimal index $\N\subset\M$ are \emph{determined} by the matrix dimension $D$ of the inclusion. The minimal index $c$ equals the squared norm of $D$, and the left and right states $\omega_l$, $\omega_r$ on the centers of $\N$ and $\M$, respectively, are uniquely determined by Perron-Frobenius theory. The entries $\lambda_{ij}$ of the matrix $\Lambda$ associated with the minimal conditional expectation $E^0$ are determined by $D$ via (\ref{eq:DdeterminesLambda}). Hence $E^0$ itself and the spherical state $\omega_s$ are determined via the formulas (\ref{eq:decompositionofE}) and (\ref{eq:sphericalstate}) by $D$ and by the unique minimal expectations  $E^0_{ij}$ in the subfactors $\N_{ij}\subset\M_{ij}$ for every $i,j$.
\end{remark}

\begin{remark}
The eigenvalue equations (\ref{eq:PF-eqns}) of Theorem \ref{thm:minimalconnectedfindim} are similar to those appearing in \cite[\Thm 3.3.2]{Jon83}. Actually they coincide with the latter in the special case of multi-matrix inclusions as we shall see in Section \ref{sec:multi-mats}. However, the normalization of the respective eigenvectors, which fixes $\mu^{1/2}$ and $\nu^{1/2}$ in our case, does not coincide with the normalization chosen in the case of algebras with a trace, \cf the proof of Corollary \ref{cor:sphericalvsmarkov}.
\end{remark}

A similar interplay between index theory and Perron-Frobenius theory appears also in \cite{GdHJ89} and \cite[\Sec 3]{Hav90}, \cite{Jol90}, together with the previously mentioned work of Jones \cite{Jon83}, where inclusions of \emph{finite} von Neumann algebras again with finite-dimensional centers (typically finite direct sums of type $\II_1$ factors, see comments after \cite[\Prop 3.5.4]{GdHJ89}) are taken into account. In that case, there is another notion of index for a subfactor (or inclusion) $\N\subset\M$ which is defined by means of the unique normal faithful normalized (Markov) \emph{trace} on $\M$, see \cite[\Def 3.7.5]{GdHJ89} and \cite[\Thm 3.9]{Hav90}. This notion of index need not coincide with the minimal index, beyond the class of the so-called \emph{extremal} inclusions, see, \eg, \cite[\Sec 4]{PiPo86}, \cite{PiPo91}, \cite[\Sec 2.3]{BurPhD}, and Remarks after \cite[\Thm 4.1, 5.5]{Lon89}, \cite[\Thm 3.7]{KosLec}. Notice that every irreducible subfactor, \ie, such that $\N'\cap\M = \CC\oneop$, in particular every subfactor with index (in either sense) $<4$, is automatically extremal. 

The two main differences between the index defined by a trace and the minimal index (for inclusions of finite von Neumann algebras) will be outlined at the end of Section \ref{sec:multi-mats}. We shall come back to extremal inclusions in Section \ref{sec:s-extr}.

\section{Index for multi-matrix inclusions revisited}\label{sec:multi-mats}

The easiest class of examples of inclusions of von Neumann algebras with finite-dimensional centers are the inclusions of (finite-dimensional) \emph{multi-matrix} algebras, \ie, $\N\subset\M$ where both $\N$ and $\M$ are finite direct sums of full matrix algebras, \ie, of finite type $I$ factors. See \cite[\Ch 2]{GdHJ89} for the terminology and \cite[\Sec 3.2, 3.3]{Jon83} where the index (necessarily finite) of such inclusions has been first explicitly considered.
Denoted by $M_n(\CC)$ the algebra of $n\times n$ complex matrices, $n\in\NN$, with unity $\oneop_n$, one can show by direct computation that the minimal index of $\CC\oneop_n \subset M_n(\CC)$ is $n^2$, see, \eg, \cite[\Sec 3.3]{KosLec}, \ie, it equals the index given by the normalized trace on $M_n(\CC)$. Similarly for the most general finite type $I$ subfactor $M_m(\CC) \otimes \CC\oneop_n \subset M_{mn}(\CC) \cong M_m(\CC)\otimes M_n(\CC)$, $n,m\in\NN$, the minimal index equals $(mn)^2m^{-2} = n^2$ and the minimal conditional expectation is the partial trace. In particular, in all these cases the minimal expectation $E^0$ always coincides with the trace-preserving expectation $E_{\N,\tr}$ of \cite{Jon83}, namely $\tr(E_{\N,\tr}(x)y) = \tr(xy)$ for $x\in\M$, $y\in\N$, where $\tr$ is the unique normalized trace on the finite factor $\M$ (= $M_{mn}(\CC)$ in this case). As a consequence we have the following more general statements.

\begin{theorem}\label{thm:DisLambdaNinM}
Let $\N\subset\M$ be a connected inclusion of multi-matrix algebras with finite-dimensional centers, then the matrix dimension $D$ of $\N\subset\M$ of Definition \ref{def:dimensionmatrix} and the inclusion matrix $\Lambda_{\N}^{\M}$ of \emph{\cite[\Sec 3.2]{Jon83}} coincide.
\end{theorem}

\begin{proof}
In the notation of Theorem \ref{thm:minimalconnectedfindim}, the equality of $D$ and $\Lambda_\N^\M$ follows by the previous discussion applied entrywise to the type $I$ subfactors $\N_{ij}\subset\M_{ij}$, for every $i,j$.
\end{proof}

\begin{corollary}\label{cor:sphericalvsmarkov}
In the same situation as above, the minimal expectation $E^0\in E(\M,\N)$ is the one which preserves the unique normalized Markov trace $\tau$ of $\M$ for the pair $\N\subset\M$, see \emph{\cite[\Sec 3.3]{Jon83}}, \emph{\cite[\Sec 2.7]{GdHJ89}}, and the minimal index of $\N\subset\M$ equals $\|\Lambda_\N^\M\|^2$. 

However, the left and right states $\omega_l$ and $\omega_r$, see Theorem \ref{thm:minimalconnectedfindim}, need not coincide with the restrictions of $\tau$ to $\Z(\N)$ and $\Z(\M)$ respectively, nor the spherical state $\omega_s$, which is tracial, be equal to $\tau_{\restriction \N'\cap\M}$, \ie, we always have $\tau = \tau \circ E^0$, but $\tau = \tau \circ {E^0}'$ may or may not hold on $\N'\cap\M$.

More precisely, the equality $\omega_r = \tau_{\restriction\Z(\M)}$ is satisfied if and only if the right Perron-Frobenius eigenvector $\mu^{1/2}$ of Theorem \ref{thm:minimalconnectedfindim} is given by $\mu_i = \alpha_i^2 (\sum_k \alpha_k^2)^{-1}$, where $\M \cong \oplus_i M_{\alpha_i}(\CC)$, $\alpha_i\in\NN$ and $i=1,\ldots,m$. Similarly, $\omega_l = \tau_{\restriction\Z(\N)}$ if and only if the left Perron-Frobenius eigenvector $\nu^{1/2}$ of Theorem \ref{thm:minimalconnectedfindim} is given by $\nu_j = \beta_j^2 (\sum_k \beta_k^2)^{-1}$, where $\N \cong \oplus_j M_{\beta_j}(\CC)$, $\beta_j\in\NN$ and $j=1,\ldots,n$.
\end{corollary}

\begin{remark}\label{rmk:super-extremality}
The expression for $\mu_i$, $i=1,\ldots,m$, in the previous corollary is nothing but the ratio of the algebraic dimensions of $i$-th direct summand of $\M$ and of $\M$ itself, similarly for $\nu_j$, $j=1,\ldots,n$. In particular, these numbers do not depend on the inclusion of $\N$ in $\M$ but on the two algebras separately. 

Now, denote by $\tau^{\N\subset\M}$ the unique normalized Markov trace on $\M$ for the pair $\N\subset\M$. As a consequence of the forthcoming Proposition \ref{prop:minindexmultif}, the (minimal) index of a nested sequence of multi-matrix algebras $\L\subset\N\subset\M$ (finite-index, connected, finite-dimensional) is multiplicative whenever $\omega_r^{\L\subset\N} = {\tau^{\L\subset\N}}_{\restriction\Z(\N)}$ and $\omega_l^{\N\subset\M} = {\tau^{\N\subset\M}}_{\restriction\Z(\N)}$, namely the index of $\L\subset\M$ is the product of the indices of $\L\subset\N$ and $\N\subset\M$. Indeed, in this case, by the previous Corollary \ref{cor:sphericalvsmarkov} we have that $\omega_r^{\L\subset\N} = \omega_l^{\N\subset\M}$ and that they depend only on $\N$, hence the sufficient condition for multiplicativity expressed in Proposition \ref{prop:minindexmultif} is triggered. Moreover, the index of multi-matrix inclusions is \emph{not} always multiplicative, see Corollary \ref{cor:minnomultifLMfactors} and \cite[\Rmk \p 62]{GdHJ89}, hence we know that the spherical (tracial) state and the normalized Markov trace on $\N'\cap\M$ may well be different in general.
\end{remark}

\begin{proof} (of Corollary \ref{cor:sphericalvsmarkov}).
The equality between the minimal index of $\N\subset\M$, see Definition \ref{def:minimalindex}, and $\|\Lambda_\N^\M\|^2$ follows immediately from Theorem \ref{thm:DisLambdaNinM} and \ref{thm:minimalconnectedfindim}.

Concerning the conditional expectations, by the deep results on the index of multi-matrix inclusions due to \cite[\Thm 2.1.1, 2.1.4]{GdHJ89} we know that $\|\Lambda_\N^\M\|^2$ equals the modulus of the unique normalized Markov trace $\tau$. The latter also equals $\|\Ind(E_{\N,\tau})\|$, where $E_{\tau,\N}\in E(\M,\N)$ is the expectation that preserves $\tau$. Indeed by \cite[\Prop 3.2]{Hav90}, $\Ind(E_{\N,\tau})$ is a scalar multiple of the identity (a property that is not automatic if we consider arbitrary traces over direct sums of finite factors, actually it is equivalent to $\tau$ being a Markov trace), thus we conclude that $E^0 = E_{\N,\tau}$.

The left and right states $\omega_l$ and $\omega_r$ are described by the unique normalized Perron-Frobenius eigenvectors $\nu^{1/2}$ and $\mu^{1/2}$, respectively, via $\mu_i = \omega_r(p_i)$ and $\nu_j = \omega_l(q_j)$. The Markov trace $\tau$ is also characterized by Perron-Frobenius theory in a similar fashion. Namely, denoted by $s_i := \tau(p_i)$ and $t_j:= \tau(q_j)$ the entries of the \lqq trace vectors" (which uniquely determine the trace $\tau$ on $\M$ and $\N$), by \cite[\Prop 2.7.2]{GdHJ89} with suitable modification of the notation, \cf \cite[\Rmk \p 175]{GdHJ89}, we have that $\Lambda_\N^\M(\Lambda_\N^\M)^t$ and $(\Lambda_\N^\M)^t\Lambda_\N^\M$ admit respectively $(s_1 \alpha_1^{-1},\ldots, s_m \alpha_m^{-1})^t$ and $(t_1 \beta_1^{-1}, \ldots, t_n \beta_n^{-1})^t$ as Perron-Frobenius eigenvectors. Thus by uniqueness there are positive normalization constants $\gamma, \eta$ such that $\mu_i^{1/2} = \gamma s_i \alpha_i^{-1}$ and $\nu_j^{1/2} = \eta t_j \beta_j^{-1}$. Observe that $\sum_i s_i = 1$ and $\sum_j t_j = 1$. One can directly check that $\mu_i = s_i$ is equivalent to $\mu_i = \alpha_i^2 (\sum_k \alpha_k^2)^{-1}$, and to $s_i = \alpha_i^2 (\sum_k \alpha_k^2)^{-1}$, analogously for $\nu_j = t_j$, and the proof is complete.
\end{proof}

\begin{remark}
We mention that in the case of inclusions of finite von Neumann algebras (not necessarily multi-matrices) Havet \cite[\Thm 3.9]{Hav90}, see also \cite[\Sec 3.7]{GdHJ89}, has shown that the Markov trace-preserving expectation has scalar index and it is minimal among those expectations in $E(\M,\N)$ that preserve some faithful normal trace on $\M$.
\end{remark}

We conclude this section with a comment on the inclusions $\N\subset\M$ of finite direct sums of type $\II_1$ factors, treated in \cite[\Ch 3]{GdHJ89}. In that case, the analogue of \cite[\Thm 2.1.1]{GdHJ89}, valid for arbitrary multi-matrix inclusions, holds for type $\II_1$ inclusions with index $\leq 4$. Namely the index of $\N\subset\M$ (in the sense of the trace) equals $\|\Lambda_\N^\M\|^2$, where $\Lambda_\N^\M$ is the index matrix considered in \cite[\Def 3.5.3]{GdHJ89}, at least when it takes values in the discrete series between 1 and 4, see \cite[\Thm 3.7.13]{GdHJ89}.

On the other hand, the minimal index of $\N\subset\M$ \emph{always} equals $\|D\|^2$ by Theorem \ref{thm:minimalconnectedfindim}, hence it generalizes some aspects of the index theory from multi-matrix algebras to more general inclusions of von Neumann algebras (not necessarily finite-dimensional, nor finite).

Another important point is that the minimal index admits a purely \emph{categorical} formulation based on the existence of \lqq conjugate objects" (sometimes called \lqq duals" or also \lqq adjoints"), as we shall see in Section \ref{sec:2-cstar-cats} when dealing with 2-\Cstar-categories with finite-dimensional centers. This allows to talk about \lqq minimal index" in different contexts, that may have in principle nothing to do with inclusions of von Neumann algebras.

\section{Extremal and super-extremal inclusions}\label{sec:s-extr}

In this section we turn our attention to finite-index inclusions of \emph{finite} von Neumann algebras (not necessarily multi-matrices, but possessing a normal faithful trace). In that case, assuming connectedness and finite-dimensionality of the centers, we have two canonical tracial states on $\N'\cap\M$ obtained from the inclusion $\N\subset\M$, namely the spherical state $\omega_s$ defined in Theorem \ref{thm:minimalconnectedfindim}, and the restriction of the unique normalized Markov trace $\tau$ from $\M$ to $\N'\cap\M$ defined in \cite[\Sec 2.7, 3.7]{GdHJ89}.
Motivated by the statements of Corollary \ref{cor:sphericalvsmarkov} about multi-matrix inclusions, we give the following

\begin{definition}\label{def:superextremal}
Let $\N\subset\M$ be an inclusion of finite von Neumann algebras with finite index, finite-dimensional centers and connected. We call $\N\subset\M$ \textbf{extremal} if $E^0 = E_{\N,\tau}$, where $E^0$ is the minimal expectation and $E_{\N,\tau}$ is the Markov trace-preserving expectation in $E(\M,\N)$. Furthermore, we call $\N\subset\M$ \textbf{super-extremal} if $E^0 = E_{\N,\tau}$ and if in addition $\omega_s = \tau_{\restriction \N'\cap\M}$.
\end{definition}

\begin{lemma}\label{lem:superextremalisleft}
If $\N\subset\M$ is extremal, then super-extremality, \ie, the equality $\omega_s = \tau_{\restriction \N'\cap\M}$, is equivalent to $\omega_l = \tau_{\restriction \Z(\N)}$. Thus $\omega_l = \tau_{\restriction \Z(\N)}$ implies $\omega_r = \tau_{\restriction \Z(\M)}$, but vice versa does not hold in general. Recall that $\omega_l$ and $\omega_r$ are the left and right state for the inclusion defined in Theorem \ref{thm:minimalconnectedfindim} on the respective center. In particular, every extremal inclusion $\N\subset\M$ where $\N$ is a factor, is automatically super-extremal.
\end{lemma}

\begin{proof}
By assumption we have $E^0 = E_{\N,\tau}$. The equality $\omega_s = \tau$ on $\N'\cap\M$ trivially implies $\omega_l = \tau$ on $\Z(\N)$, vice versa follows from $\omega_s = \omega_s \circ E^0 = \tau \circ E_{\N,\tau} = \tau$ on $\N'\cap\M$. The proof that $\omega_r = \tau_{\restriction \Z(\M)}$ does not imply $\omega_l = \tau_{\restriction \Z(\N)}$ is postponed at the end of this section by means of an example.
\end{proof}

In the terminology of Definition \ref{def:superextremal}, the previous Corollary \ref{cor:sphericalvsmarkov} shows that connected multi-matrix inclusions are extremal, but not necessarily super-extremal. Moreover, as already observed in Remark \ref{rmk:super-extremality}, we have the following result on the \emph{multiplicativity} of the minimal index, that should be compared with the general case studied in Proposition \ref{prop:minindexmultif}. 

\begin{proposition}\label{prop:multipsemm}
The index, or equivalently the scalar dimension, of connected and super-extremal multi-matrix inclusions with finite-dimensional centers is multiplicative.
\end{proposition}

Now we give another characterization of super-extremality, again in the case of multi-matrix inclusions, which makes it easier to produce examples. Given $\N\subset\M$ as in Corollary \ref{cor:sphericalvsmarkov}, we denote by $\alpha$ and $\beta$ the vectors with entries in $\NN$ such that $\M\cong \oplus_i M_{\alpha_i}(\CC)$, $i=1,\ldots,m$, and $\N\cong \oplus_j M_{\beta_j}(\CC)$, $j=1,\ldots,n$. Then it is well known that $D \beta = \alpha$ is always fulfilled, where $D=\Lambda_\N^\M$ is the matrix dimension, \ie, the inclusion matrix of $\N\subset\M$ by Theorem \ref{thm:DisLambdaNinM}, and that this is the only constraint on the Bratteli diagram \cite{Bra72} constructed from the triple $(D, \alpha, \beta)$ in order to have a well-defined associated multi-matrix inclusion, see, \eg, \cite[\Prop 2.3.9]{GdHJ89}.

\begin{proposition}\label{prop:superextremalmultimat}
A multi-matrix inclusion $\N\subset\M$ given as in Corollary \ref{cor:sphericalvsmarkov} is super-extremal if and only if, in addition to $D\beta = \alpha$, the equality $D^t \alpha = d^2 \beta$ holds, where $d=\|D\|$ is the scalar dimension of the inclusion, \ie, $d^2$ is its index.

If this is the case, we have that 
$$d^2 = \frac{\|\alpha\|^2}{\|\beta\|^2}$$
where $\|\cdot\|$ is the $l^2$-norm of the vectors $\alpha$ and $\beta$ with positive integer entries. In other words, the index of the inclusion in this special case is the ratio of the algebraic dimensions of $\M$ and $\N$ (similarly to the finite type $I$ subfactor case), but here $d^2\in\NN$ (not necessarily the square of an integer), and every possible value arises from some connected super-extremal inclusion of multi-matrices.
\end{proposition}

\begin{example}
The multi-matrix inclusion with matrix dimension, \ie, inclusion matrix $D = \left( \begin{array}{cc} 1 & 1 \end{array}\right)$ and vectors $\beta = \left( \begin{array}{c} 2 \\ 2 \end{array} \right)$, $\alpha = 4$, namely $M_2(\CC)\oplus M_2(\CC) \hookrightarrow M_4(\CC)$ with the embedding $(a,b) \mapsto a\oplus b$, is super-extremal. Its index is $d^2 = 2$, hence $d=\sqrt{2}$, moreover the ratio of the dimensions is also $16/8 = 2$, in accordance with the previous proposition. 

The inclusions with matrix dimensions $D = \left( \begin{array}{cc} 1 & 1 \\ 1 & 1 \end{array}\right)$ and $D = \left( \begin{array}{cc} 2 & 0 \\ 0 & 2 \end{array}\right)$, and with the same vectors $\beta = \left( \begin{array}{c} 2 \\ 2 \end{array} \right)$, $\alpha = \left( \begin{array}{c} 4 \\ 4 \end{array} \right)$, namely two different embeddings of $M_2(\CC)\oplus M_2(\CC) \hookrightarrow M_4(\CC) \oplus M_4(\CC)$, are both super-extremal, but only the first is connected (hence strictly speaking the second should be disregarded).

The inclusion with the same matrix dimension $D = \left( \begin{array}{cc} 1 & 1 \end{array}\right)$ as before, but with vectors $\beta = \left( \begin{array}{c} 2 \\ 3 \end{array} \right)$, $\alpha = 5$, namely $M_2(\CC)\oplus M_3(\CC) \hookrightarrow M_5(\CC)$ again with the diagonal embedding, is not super-extremal as one can immediately check. Indeed, $2 \neq 25/13$. It will become clear in the proof of Proposition \ref{prop:superextremalmultimat} given below that this example also shows that $\omega_r = \tau_{\restriction \Z(\M)}$ does not imply $\omega_l = \tau_{\restriction \Z(\N)}$, as claimed in Lemma \ref{lem:superextremalisleft}. Namely, the first condition is equivalent to the Perron-Frobenius eigenvalue equation $DD^t \alpha = d^2 \alpha$, which in this case is satisfied by any number, in particular by $\alpha$, because $DD^t = d^2 = 2$, while the second condition is equivalent to the Perron-Frobenius eigenvalue equation $D^tD \beta = d^2 \beta$, which does not hold in this case.

An example of a super-extremal multi-matrix inclusion with index $d^2=4$ and a $5 \times 3$ inclusion matrix appears in \cite[\p 58]{GdHJ89}, where the consequences of super-extremality discussed in Proposition \ref{prop:superextremalmultimat} have also been accidentally noticed.
\end{example}

\begin{proof} (of Proposition \ref{prop:superextremalmultimat}).
By Corollary \ref{cor:sphericalvsmarkov} we know that connected multi-matrix inclusions are extremal, and by Lemma \ref{lem:superextremalisleft} we know that super-extremality is equivalent to $\omega_l = \tau_{\restriction \Z(\N)}$. Again by Corollary \ref{cor:sphericalvsmarkov}, in the case of multi-matrices, the last condition is equivalent to $\nu^{1/2} = \beta \|\beta\|^{-1}$, \ie, to $D^tD \beta = d^2 \beta$, where $d=\|D\|$, hence to $D^t \alpha = d^2 \beta$. By applying $D$ to both members of the last equation we get $DD^t \alpha = d^2 \alpha$, \ie, $\mu^{1/2} = \alpha \|\alpha\|^{-1}$, or equivalently $\omega_r = \tau_{\restriction \Z(\M)}$ by Corollary \ref{cor:sphericalvsmarkov}. In particular we obtained an alternative, linear-algebraic, proof of the fact that $\omega_l = \tau_{\restriction \Z(\N)}$ implies $\omega_r = \tau_{\restriction \Z(\M)}$, in the special case of multi-matrices.

The equality $d=\|\alpha\|\|\beta\|^{-1}$ is obtained by combining $D\beta = \alpha$ and $D\nu^{1/2} = d \mu^{1/2}$. The fact that $d^2\in\NN$ holds because $d^2$ is an algebraic integer and rational, thus integer, and the fact that every possible value is realized can be seen already from the examples above.
\end{proof}

In the case of finite \emph{subfactors} with finite index, see \cite[\Sec 2]{Hia88}, \cite[\Sec 4]{PiPo86}, extremality and super-extremality are the same notion. Indeed in this case $\omega_s = \omega_s \circ {E_0}_{\restriction \N'\cap\M} = {E_0}_{\restriction \N'\cap\M}$ and $\tau_{\restriction \N'\cap\M} = \tau \circ {E_{\tau,\N}}_{\restriction \N'\cap\M} = {E_{\tau,\N}}_{\restriction \N'\cap\M}$, hence $E_0 = E_{\tau,\N}$ on $\M$ trivially implies that $\omega_s = \tau$ on $\N'\cap\M$, while vice versa holds because the restriction of conditional expectations to the relative commutant is a bijection due to \cite[\Thm 5.3]{CoDe75}.

\section{Characterization of minimality ($\N$ or $\M$ factor)}\label{sec:ifNorMfactor}

In this section we give another characterization of minimality for a conditional expectation $E\in E(\M,\N)$, and its dual expectation $E'\in E(\N',\M')$, in the case where $\N\subset \M$ is a finite-index inclusion and either $\N$ or $\M$ is a factor. As a consequence, the other algebra has finite-dimensional center and the inclusion is obviously connected. In this case, we improve Theorem \ref{thm:minimalconnectedfindim} $(b)$ in the spirit of the rigid tensor \Cstar-categorical characterization of minimality given by \cite[\Thm 3.8, \Lem 3.9 and \Thm 3.11]{LoRo97} and called \emph{sphericality} in the literature (see, \eg, \cite[\Sec 4.7]{EGNO15}) valid in the subfactor case, \ie, for tensor categories with simple unit object. 

\begin{lemma}\label{lem:onefactorhencestates}
Let $\N\subset\M$ be a finite-index inclusion of von Neumann algebras and assume that either $\N$ or $\M$ is a factor. Then for every $E\in E(\M,\N)$ the restrictions of $EE'_{\restriction \N'\cap\M}$ and $E'E_{\restriction \N'\cap\M}$ give rise to two normal faithful states on the relative commutant $\N' \cap \M$.
\end{lemma}

\begin{proof}
Write the chains of inclusions $\N\cap\M'\subset\N\cap\N'\subset\M\cap\N'$ and $\N\cap\M'\subset\M\cap\M'\subset\M\cap\N'$, and observe that $E'$ maps $\N'\cap\N$ onto $\M'\cap\M$ because $E' = \Ind(E)^{-1} E^{-1}$, $\Ind(E)\in\Z(\M)$ and $E^{-1}$ is $\M'$-bimodular. Similarly, $E$ maps $\M'\cap\M$ onto $\N'\cap\N$. Now, by our factoriality assumption either $E'E = E$ and $EE'$ are states on $\N'\cap\M$, or $EE' = E'$ and $E'E$ are states on $\N'\cap\M$, more precisely they equal $\varphi_1(\cdot) \oneop$ and $\varphi_2(\cdot) \oneop$ where $\varphi_1$ and $\varphi_2$ are states, hence the desired statement.
\end{proof}

\begin{remark}
More generally, we observe that under the connectedness assumption one has $\N\cap\M' = \Z(\N)\cap\Z(\M) = \CC\oneop$. However $E$ and $E'$ do not preserve in general $\M'$ and $\N$, respectively, but only $\N'$ and $\M$. Hence $EE'$ and $E'E$ are not states in general of the relative commutant, without further assumptions, \cf equation (\ref{eq:sphericalstate}) in the proof of Theorem \ref{thm:minimalconnectedfindim}.   
\end{remark}

\begin{proposition}\label{prop:minimaliffcommute}
Let $\N\subset\M$ be a finite-index inclusion of von Neumann algebras and assume that either $\N$ or $\M$ is a factor. Then $E$ is the minimal conditional expectation in $E(\M,\N)$ if and only if
\begin{equation}\label{eq:leftrightcommut}
{EE'} = {E'E}\quad \text{on} \quad \N'\cap\M.
\end{equation}
In this case, equation (\ref{eq:leftrightcommut}) defines the spherical state $\omega_s$ appearing in Theorem \ref{thm:minimalconnectedfindim}.
\end{proposition}

\begin{proof}
The existence of a normal faithful $E$-invariant and $E'$-invariant (\lqq spherical") state $\omega_s$ on $\N'\cap\M$, which is then explicitly given as in equation (\ref{eq:sphericalstate}), boils down to (\ref{eq:leftrightcommut}) because $EE'$ and $E'E$ already define two states on $\N'\cap\M$ by Lemma \ref{lem:onefactorhencestates}. Hence $\omega_s(EE'(x))\oneop = EE'(x)$ and $\omega_s(E'E(x))\oneop = E'E(x)$ for every $x\in\N'\cap \M$, and we have the claim as a special case of Theorem \ref{thm:minimalconnectedfindim}.
\end{proof}

\section{Multiplicativity}\label{sec:multiplicativity}

Let $\L\subset \N\subset \M$ be finite-index connected inclusions of von Neumann algebras with finite-dimensional centers and denote respectively by $\{r_k, k = 1,\ldots,l\}$, $\{q_j, j=1,\ldots, n\}$, $\{p_i, i=1,\ldots,m\}$ the minimal central projections of $\L$, $\N$, $\M$. 

\begin{lemma}
If $\L\subset\N$ or $\N\subset\M$ is connected, then $\L\subset\M$ is also connected. The same conclusion holds if $\N$ is a factor, or if both $\M$ and $\L$ are factors. 
\end{lemma}

\begin{proof}
It is enough to observe that $\Z(\L)\cap\Z(\M) \subset \Z(\N)$.
\end{proof}

The following theorem says that the matrix dimension, see Definition \ref{def:dimensionmatrix}, is indeed \emph{multiplicative}, \ie, it respects inclusions.

\begin{theorem}\label{thm:dimensionmatmult}
Let $D_1$ be the $m\times n$ matrix dimension of $\N\subset\M$ and $D_2$ the $n\times l$ matrix dimension of $\L\subset\N$, then the $m\times l$ matrix dimension of $\L\subset\M$ is given by $D_3 = D_1 D_2$.
\end{theorem}

\begin{proof}
For a subfactor $\R\subset\cS$ we denote by $[\cS:\cR]$ the minimal index, previously denoted by $c$ or $c_{ij}$, so to keep better track of the inclusions. Then 
$$(D_3)_{i,k} = [\M_{ik}:\L_{ik}]^{1/2}=\sum_j [\M_{ijk}:\L_{ijk}]^{1/2}=\sum_j[\M_{ijk}:\N_{ijk}]^{1/2}[\N_{ijk}:\L_{ijk}]^{1/2}$$
where, \eg, $\M_{ijk}$ denotes $q_j\M_{ik}q_j = q_jr_k\M p_i q_j r_k$, $q_j\in \Z(\N)\subset\L'\cap\M$ are such that $\sum_j q_j = \oneop$, and we have used the additivity of the minimal index \cite[\Thm 5.5]{Lon89} and its multiplicativity \cite[\Cor 2.2]{Lon92} in the case of subfactors. Now, $p_i$ is a projection in the commutant of the subfactor $\L_{jk}\subset\N_{jk}$, hence it has central support $\oneop$ both in $\N_{jk}$ and $\L_{jk}$, thus we have $[\N_{ijk}:\L_{ijk}] = [\N_{jk}:\L_{jk}]$, and similarly $[\M_{ijk}:\N_{ijk}] = [\N'_{ijk}:\M'_{ijk}] = [\N'_{ij}:\M'_{ij}] = [\M_{ij}:\N_{ij}]$, from which we get
$$(D_3)_{i,k} = \sum_j (D_1)_{i,j} (D_2)_{j,k}$$
and the proof is complete. 
\end{proof}

\begin{remark}
A straightforward computation shows that the expectation matrix $\Lambda=\Lambda_E$ of an arbitrary conditional expectation $E$ (not necessarily the minimal one) is also multiplicative, \ie, $\Lambda_{E} = \Lambda_F\Lambda_G$ whenever $E=GF$.
\end{remark}

Concerning the minimal index, \ie, the spectral radius of the matrices $D^t D$ and $D D^t$ thanks to Theorem \ref{thm:minimalconnectedfindim} when we consider connected inclusions, from Theorem \ref{thm:dimensionmatmult} we deduce the following.

\begin{corollary}\label{cor:minmultifNfactor}
Let $\L\subset\N\subset\M$ as in Theorem \ref{thm:dimensionmatmult}. If $\N$ is a factor, then the minimal index $c_3$ of $\L\subset\M$ is given by $c_3=c_1c_2$, where $c_1$ is the minimal index of $\N\subset\M$ and $c_2$ is the minimal index of $\L\subset\N$. 
\end{corollary}

\begin{proof}
The spectral radius of a product of matrices (not necessarily square) is invariant under cyclic permutations, see \cite[\Thm 1.3.22]{HoJoBook}, hence $c_3$ equals the spectral radius of $D_1^tD_1D_2D_2^t$, where $D_1^tD_1$ and $D_2D_2^t$ are both $n\times n$ matrices (positive as operators, with non-negative entries). By factoriality assumption on $\N$ we have that $n=1$, hence the spectral radius of the product trivially equals the product of the respective spectral radii, \ie, $c_3 =c_1c_2$.
\end{proof}

\begin{corollary}\label{cor:minnomultifLMfactors}
Let $\L\subset\N\subset\M$ as in Theorem \ref{thm:dimensionmatmult}. If $\L$ and $\M$ are factors, then $c_3 = \cos^2(\alpha) c_1c_2$, where $c_1$, $c_2$ and $c_3$ are as above and $\alpha$ is the angle between $D_1$ and $D_2$ as vectors in $\RR^n$.
\end{corollary}

\begin{proof}
Observe that $D_1D_2$ is a number because $m=1$ and $l=1$ by factoriality assumptions on $\M$ and $\L$, and use again the invariance under cyclic permutations of the spectral radius, or directly $D_3 = D_1D_2$. Then $c_3 = (D_1D_2)^2 = [(D_1|D_2)]^2$ and $c_1 = \|D_1\|^2$, $c_2 = \|D_2\|^2$, where $\langle\cdot|\cdot\rangle$ and $\|\cdot\|$ indicate respectively the $l^2$-scalar product and norm in $\RR^n$. 
\end{proof}

The analysis of these special cases already shows that the minimal index itself is \emph{not} always multiplicative once we deal with inclusions of von Neumann algebras with non-trivial centers, \cf with the case of subfactors \cite{Lon92} where multiplicativity always holds.

More generally, the minimal index is only \emph{submultiplicative} (this is not true for the spectral radius in general) and we state a sufficient condition 
for the minimal index of $\L\subset\M$, to be the product of the minimal indices of two intermediate inclusions $\L\subset\N$ and $\N\subset\M$. The same condition guarantees that the minimal expectation $G^0\in E(\M,\L)$ factors  through  $\N$, \ie, it is the product of the two intermediate minimal expectations $E^0\in E(\M,\N)$ and $F^0\in E(\N,\L)$.

\begin{proposition}\label{prop:minindexmultif}
Let $\L\subset\N\subset\M$ as in Theorem \ref{thm:dimensionmatmult}, and 
denote by $c_1$, $c_2$ and $c_3$ the minimal indices of $\N\subset\M$, $\L\subset\N$ and $\L\subset\M$ as above, then $c_3 \leq c_1 c_2$.

If the left Perron-Frobenius eigenvector of $\N\subset\M$ coincides with the right Perron-Frobenius eigenvector of $\L\subset\N$, or equivalently if the left state $\omega_l^{\N\subset\M}$ for $\N\subset\M$ and the right state $\omega_r^{\L\subset\N}$ for $\L\subset\N$ are the same state on $\Z(\N)$, then $c_3 = c_1 c_2$. Moreover, we have that $G^0 = F^0 E^0$.
\end{proposition}

\begin{proof}
In the same notation as in Theorem \ref{thm:dimensionmatmult}, let $D_1$ and $D_2$ be the matrix dimensions of $\N\subset\M$ and $\L\subset\N$ respectively, hence $D_3 = D_1 D_2$ is the matrix dimension of $\L\subset\M$. By Theorem \ref{thm:minimalconnectedfindim} the minimal index of $\L\subset\M$ equals the spectral radius of $D_1^tD_1D_2D_2^t$, where we have used again invariance under cyclic permutations \cite[\Thm 1.3.22]{HoJoBook}, and $D_1^tD_1$, $D_2D_2^t$ are both $n\times n$ matrices, where $n$ is the dimension of $\Z(\N)$. Denoted by $\rho(A)$ the spectral radius of a square matrix $A$, we have that $\rho(D_1^tD_1D_2D_2^t) \leq \|D_1^tD_1D_2D_2^t\| \leq \|D_1^tD_1\|\|D_2D_2^t\| = \rho(D_1^tD_1)\rho(D_2D_2^t)$, hence the submultiplicativity of the minimal index. Alternatively, one can simply observe that $\|D_3\| \leq \|D_1\|\|D_2\|$, \ie, $d_3\leq d_1d_2$, where $d_i = c^{1/2}_i$, $i=1,2,3$, is the scalar dimension of the respective inclusion.  

If $D_1^tD_1$ and $D_2D_2^t$ have the same Perron-Frobenius eigenvector, then the spectral radius of the product is the product of the spectral radii by \cite[\Thm 8.3.4]{HoJoBook}, see also \cite[\Sec 3.C]{JoBr90}, hence $c_3 = c_1c_2$ in this case.
In particular, this already implies that the minimal expectation in $E(\M,\L)$ is the one given by the composition product $F^0 E^0$ (because $F^0$ and $E^0$ have scalar index, see Remark \ref{rmk:notevenIndFEismult}), but we want to exhibit the associated spherical state $\omega_s^{\L\subset\M}$ on $\L'\cap\M$ appearing in Theorem \ref{thm:minimalconnectedfindim}. The natural ansatz is $\omega_s^{\L\subset\M}:=\omega_s^{\L\subset\N}\circ E^0$ or $\omega_s^{\L\subset\M}:=\omega_s^{\N\subset\M}\circ {F^0}'$, in order to have respectively $F^0E^0$-invariance or ${E^0}'{F^0}'$-invariance on $\L'\cap\M$ (observe that $(F^0E^0)' = {E^0}'{F^0}'$ again because $F^0$ and $E^0$ have scalar index). Indeed, for the first invariance use the fact that $E^0$ and $F^0$ both preserve $L'$ together with $F^0$-invariance of $\omega_s^{\L\subset\N}$ on $\L'\cap\N$, similarly for the second invariance using ${E^0}'$-invariance of $\omega_s^{\N\subset\M}$ on $\N'\cap\M$. To show that the two actually ansatzes coincide we first prove that $E^0{F^0}' = {F^0}'E^0$ on $\L'\cap\M$, 
\footnote{This statement, rephrased in the language of rigid tensor \Cstar-categories, is the key ingredient in the proof that sphericality is preserved by the composition product \cite[\Thm 3.8]{LoRo97} in the case of simple tensor unit (subfactor case). 
The proof presented here, inspired by the one of \cite[\Thm 3]{KaWa95} in the trivial center case, is not categorical but rather tailored to the theory of index for conditional expectations.}. 
For convenience, by tensoring with a type $I_\infty$ factor we may assume that our algebras are properly infinite, hence the finiteness of the index of either expectation guarantees the existence of a Pimsner-Popa basis made of a single element in the algebra, see \cite[\Lem 3.21, \Prop 3.22]{BDH88} and \cite[\Def 1.1.4]{Pop95} for the definition. 
Namely, \eg, for $F^0\in E(\N,\L)$ we have $v\in \N$ such that $v^* e v = \oneop$, where $e$ is a Jones projection implementing $F^0$, hence $v^*v = \Ind(F^0)$, and we may assume $F^0(vv^*) = \oneop$. By \cite[\Rmk 3.8]{BDH88}, see also \cite[\Prop 1.5]{Jol91}, \cite[\Rmk A]{KaWa95}, we can express the dual conditional expectation as ${F^0}' = \Ind(F^0)^{-1} v^* \cdot v$ on $\L'$, hence the desired commutation relation between $E^0$ and ${F^0}'$ immediately follows from $E^0(v) = v$. Now, by the remaining two invariance properties of $\omega^{\L\subset\N}_s$ and $\omega^{\N\subset\M}_s$ and by the assumption that they coincide on $\Z(\N)$, we have the spherical state $\omega_s^{\L\subset\N}\circ E^0 = \omega_s^{\N\subset\M}\circ {F^0}' =: \omega_s^{\L\subset\M}$ for $F^0E^0$, thus concluding the proof.
\end{proof}

\begin{remark}\label{rmk:notevenIndFEismult}
In general it is also not true that the index of conditional expectations respects the composition product, \ie, for $E\in E(\M,\N)$, $F\in E(\N,\L)$ and $\L\subset\N\subset\M$ it might be that $\Ind(FE) \neq \Ind(F)\Ind(E)$, actually the right hand side need not be an element of $\Z(\M)$. However, equality holds whenever $\Ind(F)$ is also an element of  $\Z(\M)$, see \cite[\Prop 3.14]{BDH88}, which is always the case, \eg, when $\N$ is a factor, \cf Corollary \ref{cor:minmultifNfactor} above.
\end{remark}

The next corollary shows that the minimal index is multiplicative in the Jones tower of basic extensions.

\begin{corollary}\label{cor:jonesextensionmultip}(Jones basic extension case, \cf \emph{\cite{KoLo92}}).
Let $\N\subset\M$ be a finite-index connected inclusion of von Neumann algebras with finite-dimensional centers, let $E^0\in E(\M,\N)$ be the minimal expectation and consider the Jones extension $\N\subset\M\subset\M_1 = \M \vee \{e\}$, where $e$ is a Jones projection implementing $E^0$. Then the matrix dimension $\tilde D$ of $\M\subset\M_1$ is the transpose of the matrix dimension $D$ of $\N\subset\M$, \ie, $\tilde D = D^t$, the inclusions $\N\subset\M$ and $\M\subset\M_1$ have the same minimal index $c$, and $c^2$ is the minimal index of $\N\subset\M_1$. Moreover, $E^0 (E^0)_1$ is the minimal expectation for $\N\subset\M_1$, where $(E^0)_1 := \Ad_J \circ {E^0}' \circ \Ad_J \in E(\M_1,\M)$ and $J$ is a modular conjugation of $\M$.

Similar statements hold for the Jones tower of basic extensions $\N\subset\M\subset\M_1\subset\M_2\subset\ldots$ constructed from the minimal expectation $E^0\in E(\M,\N)$.
\end{corollary}

\begin{proof}
The inclusions $\M'\subset\N'$ and $\M\subset\M_1$ are anti-isomorphic by means of the modular conjugation $J$ of $\M$, where we may assume $\M$ to be in standard form, see, \eg, \cite[\Sec 1.1.3]{Pop95}. Hence $\tilde D = D^t$ and the rest is a consequence of the results of this section. The second statement is obtained by iterating the argument.
\end{proof}

\section{Additivity}\label{sec:additivity}

Let $\N\subset\M$ be a finite-index connected inclusion of von Neumann algebras with finite-dimensional centers, and denote again by $\{q_j, j=1,\ldots, n\}$, $\{p_i, i=1,\ldots,m\}$ the minimal central projections of $\N$ and $\M$ respectively. We now show that the matrix dimension of the inclusion, see Definition \ref{def:dimensionmatrix}, is \emph{additive} with respect to an arbitrary partition of unity by projections in the relative commutant.

\begin{proposition}\label{prop:dimensionmatadd}
Let $f_1,\ldots,f_N$ be mutually orthogonal projections in $\N'\cap\M$ such that $\sum_\alpha f_\alpha = \oneop$. Define $\N_\alpha := \N f_\alpha$, $\M_\alpha := f_\alpha \M f_\alpha$ for every $\alpha=1,\ldots,N$. Then $D = \sum_{\alpha} D_\alpha$ where $D$ is the matrix dimension of $\N\subset\M$ and $D_\alpha$ is the matrix dimension of $\N_\alpha\subset\M_\alpha$.
\end{proposition}

\begin{proof}
Assume first that the $f_\alpha$, $\alpha=1,\ldots,N$, have supports $\oneop$ both in $\Z(\N)$ and $\Z(\M)$. For every $\alpha=1,\ldots,N$ the minimal central projections of $\N_\alpha$ and $\M_\alpha$ are given by $f_\alpha q_j$, $j=1,\ldots,n$, and $f_\alpha p_i$, $i=1,\ldots,m$, respectively. Also, each of the $f_\alpha q_j$ and $f_\alpha p_i$ is not zero, otherwise the support of $f_\alpha$ in $\Z(\N)$ or $\Z(\M)$ would not be $\oneop$, in particular the dimensions of the centers do not decrease by cutting with $f_\alpha$. Hence by the additivity of the square root of the minimal index (\ie, of the intrinsic dimension) for the subfactors $\N_{ij}\subset\M_{ij}$, see \cite[\Thm 5.5]{Lon89} and \cf \cite[\p 114]{LoRo97}, for every $i,j$ we get
$$(D)_{i,j} = \sum_\alpha (D_\alpha)_{i,j}$$
where $\sum_\alpha f_\alpha p_i q_j = p_i q_j$ is a partition of unity by projections in $\N_{ij}'\cap\M_{ij}$.

For an arbitrary family of $f_\alpha$, $\alpha=1,\ldots,N$ summing up to $\oneop$, we only have to observe that more zeros might occur as entries $(D_\alpha)_{i,j}$ for some $\alpha$, think of the case $f_\alpha = 0$ for some $\alpha$, or $f_\alpha = p_i$ or $q_j$, $\alpha=1,\ldots,N$, where $N=m$ or $n$, and the same conclusion as before holds.
\end{proof}

\begin{remark}
Consider the partition of unity given by projections $f_\alpha = p_iq_j\in\N'\cap\M$, $\alpha = (i,j)$ and $\alpha = (1,1),\ldots,(m,n)$ in multi-index notation. Then $(D_{(i,j)})_{k,l} = d_{kl}$ if $(i,j)=(k,l)$ and $(D_{(i,j)})_{k,l} = 0$ otherwise, for every $k=1,\ldots,m$, $l=1,\ldots,n$, hence clearly $D = \sum_{(i,j)} D_{(i,j)}$.
\end{remark}

Contrary to the case of subfactors, the scalar dimension $d$ itself (the norm of the matrix dimension) cannot be additive in general. Otherwise, the same argument of \cite[\Thm 2.1]{Lon92} would lead to the multiplicativity of the scalar dimension in this case as well, hence it would contradict Corollary \ref{cor:minnomultifLMfactors}. Nonetheless, if $\N$ or $\M$ is a factor, the square of the scalar dimension, \ie, the minimal index by Theorem \ref{thm:minimalconnectedfindim}, is additive with respect to the decomposition of unity by means of minimal central projections in the other algebra, namely

\begin{proposition}\label{prop:additivityindex}
Assume that either $\N$ or $\M$ is a factor, let $d$ be the scalar dimension of $\N\subset\M$ and $d_k$ the entries of the (column or row) matrix dimension $D$. Then the non-trivial normalized Perron-Frobenius eigenvector of Theorem \ref{thm:minimalconnectedfindim} is a multiple either of $D$ or of $D^t$. Moreover, $d^2 = \sum_k d_k^2$, \ie, the minimal index of $\N\subset\M$ is additive with respect to the decomposition of unity by means of minimal central projections either in $\Z(\M)$ or $\Z(\N)$.
\end{proposition}

\begin{proof}
Assume that $\N$ is a factor, and let $m$ be the dimension of $\Z(\M)$ as before. Then $DD^t$ is an $m\times m$ matrix with entries $d_id_j$, $i,j=1,\ldots, m$. Using the eigenvalue equation for $DD^t$ and the expression for the scalar dimension $d$ in terms of the entries of $D$ and of the Perron-Frobenius eigenvectors contained in Theorem \ref{thm:minimalconnectedfindim}, equation (\ref{eq:weightedadditivd}), we get $(\mu_i)^{1/2} = d_i (\sum_k d_k^2)^{-1/2}$, $i = 1,\ldots, m$, and the other eigenvalue equation $D^tD = d^2$ gives the second statement. Analogously when $\M$ is a factor. 
\end{proof}

\begin{remark}
In either case ($\N$ or $\M$ factor) the eigenvalue equations (\ref{eq:PF-eqns}) in Theorem \ref{thm:minimalconnectedfindim} boil down to $DD^t x = D^tD x$, where $x$ denotes either $\mu^{1/2}$ or $\nu^{1/2}$, and on one of the two sides we mean multiplication of a scalar times a vector.

Furthermore, the formula (\ref{eq:DdeterminesLambda}) for the entries of the expectation matrix $\Lambda_{E^0}$ of the minimal expectation $E^0$ in Theorem \ref{thm:minimalconnectedfindim}, becomes in this case $\lambda_k (= \text{ either }\mu_k \text{ or } \nu_k) = d_k^2 \sum_l d_l^2$, where $k, l$ run between $1$ and either $m$ or $n$.
\end{remark}

\begin{remark}
The additivity of the minimal index (not of its square root, \ie, of the scalar dimension) with respect to minimal central projections contained in Proposition \ref{prop:additivityindex} (assuming $\N$ or $\M$ to be a factor) has already been observed in \cite[\Sec 2.3]{BKLR15}. See also \cite[\Ex 2.7]{Jol90} for the same statement about additivity concerning the index given by a trace in the case of finite von Neumann algebras.   
\end{remark}

\section{Dimension theory for rigid 2-\Cstar-categories}\label{sec:2-cstar-cats}

The theory of dimension for 2-\Cstar-categories is outlined in \cite[\Sec 7]{LoRo97}, under the assumption that the 1-arrow units are simple (which means factoriality in the context of von Neumann algebras). Motivated by the concrete examples of inclusions and bimodules among von Neumann algebras with finite-dimensional centers, for which we refer to Section \ref{sec:inclusions} and \cite[\Sec 2]{Lon17arxiv}, see also \cite[\Ch 5]{ConnesBook}, \cite{Lan01}, \cite{SaYa17arxiv} and references therein, we develop a theory of dimension for 2-\Cstar-categories, under the assumption that every 0-object has finite-dimensional \lqq center", \ie, that the space of 2-arrows between the associated unit 1-arrow and itself is finite-dimensional. 
We re-derive then all the results obtained in the previous sections in this more general (categorical) language, that we regard as the most natural to talk about dimension and minimal index.

\begin{definition}
A (strict) \textbf{2-category} $\C$, in the sense of \cite[\Ch XII]{Mac98}, is a collection of
\begin{itemize}
\item \textbf{0-objects} or \textbf{objects}: a class $\C^{(0)}$ whose elements will be denoted by $\N,\M,\L,\ldots$
\item \textbf{1-arrows} or \textbf{1-morphisms}: a class $\C^{(1)}$ whose elements will be denoted by $X,Y,Z,\ldots$, equipped with a source and a target map both from $\C^{(1)}$ to $\C^{(0)}$. If the source and the target of $X$ are respectively $\N$ and $\M$, we write $X \in \Hom_{\C^{(1)}}(\N,\M)$ or $X: \N \rightarrow \M$.

A unit map from $\C^{(0)}$ to $\C^{(1)}$ associating to every $\N$ a unit 1-arrow $I_\N:\N \rightarrow \N$ and a composition map associating to every pair $(X,Y)$ of 1-arrows that is \lqq composable", \ie, such that $X:\N\rightarrow\M$ and $Y:\M\rightarrow\L$, a composed 1-arrow $Y \otimes X:\N\rightarrow\L$, \footnote{Alternatively, we could have used the notation $X:\M\leftarrow\N$, $Y:\L\leftarrow\M$ for which the order in the composition symbol $Y\otimes X:\L\leftarrow\N$ might appear more natural, or we could have used the opposite \lqq diagrammatic" order $X\otimes Y$ motivated by $\N\xrightarrow{X}\M\xrightarrow{Y}\L$.}. 
Associativity of the composition and unit law hold strictly, namely $X\otimes I_\N = X$ and $I_\M \otimes X = X$ if $X:\N\rightarrow\M$, and $(Z\otimes Y)\otimes X = Z\otimes (Y\otimes X)$ whenever the composition is defined.
\item \textbf{2-arrows} or \textbf{2-morphisms}, or also \textbf{2-cells}:
a class $\C^{(2)}$ whose elements will be denoted by $r,s,t\ldots$, equipped with a source and a target map both from $\C^{(2)}$ to $\C^{(1)}$. Only 2-arrows between \lqq parallel" 1-arrows are allowed.
Namely, if the source and the target of $r$ are respectively $X$ and $Y$, we require that $X$ and $Y$ have a common source $\N$ and a common target $\M$, and we write $r \in \Hom_{\C^{(2)}}(X,Y)$ or $r: X \Rightarrow Y$, or better $r: X \Rightarrow Y : \N \rightarrow \M$. 

A unit map from $\C^{(1)}$ to $\C^{(2)}$ associating to every $X$ a unit 2-arrow $\oneop_X:X \Rightarrow X$ and a \lqq vertical" composition map associating to every pair $(r,s)$ of 2-arrows that is \lqq vertically composable", \ie, such that $r:X\Rightarrow Y$ and $s:Y\Rightarrow Z$, a composed 2-arrow $s \cdot r:X\Rightarrow Z$. As before, associativity of the vertical composition and vertical unit law hold strictly, namely $r\cdot \oneop_X = r$ and $\oneop_Y \cdot r = r$ if $r:X\Rightarrow Y$, and $(t\cdot s)\cdot r = t\cdot (s\cdot r)$ whenever the composition is defined.

A \lqq horizontal" composition map associating to every pair $(s,t)$ of 2-arrows that is \lqq horizontally composable", \ie, such that $s:X\Rightarrow Y:\N\rightarrow \M$ and $t:Z\Rightarrow T:\M\rightarrow \L$, a composed 2-arrow $t \otimes s:Z\otimes X\Rightarrow T\otimes Y:\N\rightarrow\L$. As for the composition of 1-arrows, the horizontal composition of 2-arrows is assumed to be strictly associative, namely $(t\otimes s)\otimes r = t\otimes (s\otimes r)$ whenever it is defined, and the horizontal unit law holds strictly as follows $s\otimes \oneop_{I_\N} = s, \oneop_{I_\N} \otimes r = r$, whenever it is defined. Here $\oneop_{I_\N}$ is the unit 2-arrow associated with the unit 1-arrow associated with the object $\N$.

Among the two compositions (vertical/horizontal) of 2-arrows $s,t,u,v$, we require the following \lqq interchange law" (also called \lqq middle four exchange law") to hold strictly, whenever it is defined 
$$(v\otimes u)\cdot(t\otimes s) = (v\cdot t) \otimes (u\cdot s)$$
and that the horizontal composition respects the vertical unit 2-arrows, \ie \footnote{Notice that $\oneop_X \otimes \oneop_{I_\N} = \oneop_{X} = \oneop_{X\otimes I_{\N}}$ and $\oneop_{I_\M} \otimes \oneop_X = \oneop_{X} = \oneop_{I_\M \otimes X}$ for each $X:\N\rightarrow\M$ follow already from previous requirements.} 
$$\oneop_Y \otimes \oneop_X = \oneop_{Y \otimes X}$$
whenever the composition makes sense, \footnote{These last two properties mean bifunctoriality of the horizontal composition. Notation-wise, we may drop some parentheses by evaluating $\otimes$ always before $\cdot$. We might also omit the $\cdot$ symbol and write, \eg, $sr$ for $s\cdot r$.}.
\end{itemize}
\end{definition}

\begin{remark}
Observe that, by the definition above, $\C^{(0)}$ and $\C^{(1)}$ alone form an ordinary category, and that a 2-category can be equivalently described as an assignment of a category $\Hom_{\C^{(1)}}(\N,\M)$ to every pair of objects $(\N,\M)$, together with a horizontal composition bifunctor for every triple $(\N,\M,\L)$ and horizontal unit functor for every object $\N$, see \cite{Mac98}, \cite[\Sec 2.12]{EGNO15}. 
\end{remark}

\begin{remark}
Recall that a (strict) 2-category with only one 0-object $\M$ is a strict tensor category, where the tensor multiplication is the horizontal composition bifunctor, the tensor unit is $I_\M$, and every two objects (1-arrows of the 2-category) are composable. This motivates the notation $\otimes$ for the horizontal composition. As for ordinary non-strict tensor categories a coherence theorem holds for 2-categories, namely every \lqq non-strict 2-category" (called in the literature weak 2-category or bicategory) is biequivalent to a 2-category, see, \eg, \cite[\Thm 1.5.15]{Lei04}.   
\end{remark}

The following definition appears in \cite[\Sec 7]{LoRo97}, see also \cite{GLR85}, \cite{Zit07}, \cite{BCLS17arxiv}.

\begin{definition}
A 2-category $\C$ is a \textbf{2-\Cstar-category} if
\begin{itemize}
\item $\Hom_{\C^{(2)}}(X,Y)$ is a complex Banach space for every pair of parallel 1-arrows $(X,Y)$. The vertical composition $(r,s) \mapsto s r$ and the horizontal composition $(r,s) \mapsto s\otimes r$ of 2-arrows, whenever defined, are bilinear.
\item There is an antilinear contravariant involutive *-map: $\Hom_{\C^{(2)}}(X,Y) \rightarrow \Hom_{\C^{(2)}}(Y,X)$ for every pair of parallel 
1-arrows $(X,Y)$, namely $t\mapsto t^*$ is antilinear, fulfills $(s r)^* = r^* s^*$, $t^{**} = t$, and it is in addition positive, namely if $t\in\Hom_{\C^{(2)}}(X,Y)$ then $t^* t$ is a positive element in $\Hom_{\C^{(2)}}(X,X)$, \ie, $t^* t = s^* s$, $s\in\Hom_{\C^{(2)}}(X,X)$, and unital, \ie, $\oneop_X^* = \oneop_X$. On the other hand, the horizontal composition is $^*$-preserving, namely $(s \otimes r)^* = s^* \otimes r^*$.
\item The collection of Banach norms is submultiplicative, namely $\| s r\| \leq \|s\| \|r\|$, and fulfills the \Cstar-identity $\|t^* t\| = \|t\|^2$. In particular, $t^* t = 0$ implies $t=0$, \ie, the *-map is positive. 
\end{itemize}
\end{definition}

An immediate consequence of the above assumptions is that

\begin{lemma}
In a 2-\Cstar-category, $\Hom_{\C^{(2)}}(X,X)$ is a unital \Cstar-algebra for every 1-arrow $X$, and $\Hom_{\C^{(2)}}(I_\N,I_\N)$ is in addition commutative for every object $\N$.
\end{lemma}

We shall always assume our 2-\Cstar-categories to be closed under finite \lqq direct sums of 1-arrows" and \lqq sub-1-arrows". Namely, for every pair of parallel 1-arrows $X,Y:\N\rightarrow\M$ there is a 1-arrow $Z:\N\rightarrow\M$ and two isometric 2-arrows $v:X\Rightarrow Z$ and $w:Y\Rightarrow Z$, \ie, $v^* v = \oneop_X$ and $w^* w = \oneop_Y$, such that $vv^* + ww^* = \oneop_Z$, and for every 1-arrow $T:\N\rightarrow\M$ and every projection $p:T\Rightarrow T$, \ie, $p = p^*p$, there is a 1-arrow $S:\N\rightarrow\M$ and an isometric 2-arrow $u:S\Rightarrow T$, \ie, $u^*u =\oneop_S$, such that $uu^* = p$. We shall write $Z=X\oplus Y$ and $S \prec T$.

Every 2-\Cstar-category can indeed be completed to a 2-\Cstar-category with finite direct sums and sub-1-arrows, see \cite[\App]{LoRo97}.

\begin{definition}\label{def:simpleconnectedfactorial}
In a 2-\Cstar-category, a 1-arrow $X$ is called \textbf{simple}, or \textbf{irreducible}, if 
$$\Hom_{\C^{(2)}}(X,X) = \CC \oneop_X.$$ 
If $X:\N\rightarrow\M$, notice that the maps $s \mapsto \oneop_X \otimes s$ and $t \mapsto t \otimes \oneop_X$ are unital, not necessarily faithful,
*-homomorphisms respectively from $\Hom_{\C^{(2)}}(I_\N,I_\N)$ and $\Hom_{\C^{(2)}}(I_\M,I_\M)$ both to $\Hom_{\C^{(2)}}(X,X)$, actually to its center $\Z(\Hom_{\C^{(2)}}(X,X))$, \footnote{The maps $s \mapsto \oneop_X \otimes s$ and $t \mapsto t \otimes \oneop_X$ are what Mac Lane \cite[\p 275]{Mac98} calls respectively right and left \lqq whiskering" with $X$, but on arbitrary 2-arrows $s$ and $t$ when suitably composable.}. 

We say that $X$ is \textbf{connected} if 
$$(\oneop_X \otimes \Hom_{\C^{(2)}}(I_\N,I_\N)) \,\cap\, (\Hom_{\C^{(2)}}(I_\M,I_\M) \otimes \oneop_X) = \CC\oneop_X.$$
We shall denote $\Z(\N):=\Hom_{\C^{(2)}}(I_\N,I_\N)$ and $\Z(\M):=\Hom_{\C^{(2)}}(I_\M,I_\M)$, and call them respectively the \lqq center" of $\N$ and $\M$. We call $\Hom_{\C^{(2)}}(X,X)$ the \lqq relative commutant" of $X$, and $\Z^X_l(\N) := \oneop_X \otimes \Hom_{\C^{(2)}}(I_\N,I_\N)$, $\Z^X_r(\M) := \Hom_{\C^{(2)}}(I_\M,I_\M) \otimes \oneop_X$ respectively the \lqq left center" and the \lqq right center" of $X$. Notice that $\Z(\N) = \Z^{I_\N}_{l/r}(\N) = \Hom_{\C^{(2)}}(I_\N,I_\N)$ by definition.

We say that $X$ is \textbf{factorial} if its left and right centers are both trivial, \ie, if $\Z^X_l(\N) = \CC\oneop_X$ and $\Z^X_r(\M) = \CC\oneop_X$.
\end{definition}

\begin{lemma}
In a 2-\Cstar-category $\C$, a simple 1-arrow, or more generally a 1-arrow whose relative commutant has trivial center, is necessarily connected. If $\C$ has simple unit 1-arrows (simple tensor units), \ie, if every object has trivial center, then every 1-arrow is factorial and a fortiori connected.
\end{lemma}

The maps defined by left/right horizontal composition (\lqq tensoring") with $\oneop_X$, namely $s\in\Z(\N) \mapsto \oneop_X \otimes s$ and $t\in\Z(\M) \mapsto t \otimes \oneop_X$, are compatible with the operations in $\C$ in the following sense. Let $X,Y,Z:\N\rightarrow\M$ be parallel 1-arrows in $\C$.

\begin{lemma}\label{lem:whiskerisrep}
If $w: Y \Rightarrow X$ is 2-arrow in $\C$, then $w \cdot \oneop_Y \otimes s = \oneop_X \otimes s \cdot w$ for every $s\in\Z(\N)$ and $w \cdot t \otimes \oneop_Y = t \otimes \oneop_X \cdot w$ for every $t\in\Z(\M)$ (intertwining relations between representations of the respective center). If $w$ is an isometry, hence $Y\prec X$, then $\oneop_Y \otimes (\cdot) = w^* \cdot \oneop_X \otimes (\cdot) \cdot w$ and similarly on the right (subrepresentations). If $Z = X \oplus Y$ by means of isometries $w,v$, then $\oneop_Z \otimes (\cdot) = w \cdot \oneop_X \otimes (\cdot) \cdot w^* + v \cdot \oneop_Y \otimes (\cdot) \cdot v^*$ and similarly on the right (direct sum of representations).
\end{lemma}

\begin{proof}
The first statement (which is anyway trivial if written using the graphical calculus for 2-\Cstar-categories) is an immediate consequence of $w=w\otimes\oneop_{I_\N} = \oneop_{I_\M} \otimes w$, of $w = w \cdot \oneop_Y = \oneop_X \cdot w$, and of the interchange law. The rest is straightforward.
\end{proof}

The following is the main assumption on a 2-\Cstar-category, that will allow us to define a matrix-valued dimension function on the category, or more precisely on its 1-arrows.

\begin{definition}
A 2-\Cstar-category $\C$ is said to be \textbf{rigid}, or to have \textbf{conjugate} (also called \textbf{dual}) 1-arrows, if for every 1-arrow $X:\N\rightarrow\M$ there is a 1-arrow denoted by $\Xbar:\M\rightarrow\N$ and a pair of 2-arrows denoted by $r_X\in\Hom_{\C^{(2)}}(I_\N,\Xbar\otimes X)$, $\rbar_X\in\Hom_{\C^{(2)}}(I_\M,X\otimes \Xbar)$ fulfilling the \textbf{conjugate equations}, namely
$$\rbar_X^*\otimes \oneop_X \cdot \oneop_X \otimes r_X = \oneop_X, \quad 
r_X^*\otimes \oneop_\Xbar \cdot \oneop_\Xbar \otimes \rbar_X = \oneop_\Xbar.$$
\end{definition}

\begin{remark}
In the \Cstar setting, the conjugate equations automatically entail two more equations by taking the *-map. Moreover, if $\Xbar$ is a conjugate of $X$, then $X$ is easily seen to be a conjugate of $\Xbar$ by setting $r_\Xbar := \rbar_X$ and $\rbar_\Xbar := r_X$. In the language of Mac Lane \cite[\Ch XII, \Sec 4]{Mac98}, $X$ and $\Xbar$ are \emph{adjoint} 1-arrows in $\C$, with $X$ both a left/right adjoint to the right/left adjoint $\Xbar$. 
\end{remark}

\begin{remark}
There are several ways to realize a rigid tensor \Cstar-category by means of endomorphisms or bimodules of von Neumann factors, see \cite{GiYu17arxiv} and references therein, provided the tensor unit is simple. We are not aware of similar statements in the case of non-simple units and, more generally, of rigid 2-\Cstar-categories. 
\end{remark}

\begin{lemma}\label{lem:uniqueconjandsol}
For any two conjugates $\Xbar$ and $\Xbar '$ of the same 1-arrow $X$, there is an invertible 2-arrow in $\Hom_{\C^{(2)}}(\Xbar,\Xbar ')$. 

Moreover, if $X$ and $\Xbar$ are conjugate 1-arrows, any two solutions $r_X, \rbar_X$ and $s_X, {\overline s}_X$ of the conjugate equations are obtained from one another by $s_X = w \otimes \oneop_X \cdot r_X$ and ${\overline s}_X = \oneop_X \otimes v \cdot \rbar_X$, where $v$ and $w$ are invertible 2-arrows in $\Hom_{\C^{(2)}}(\Xbar,\Xbar)$ fulfilling $v^{-1} = w^*$. A similar statement involving invertible 2-arrows in $\Hom_{\C^{(2)}}(X,X)$ holds.
\end{lemma}

\begin{proof}
Let $r_X, \rbar_X$ and $s_X, {\overline s}_X$ be solutions of the conjugate equations for $X$, $\Xbar$ and $X$, $\Xbar '$ respectively. Then $v := r_X^* \otimes \oneop_{\Xbar '} \cdot \oneop_{\Xbar} \otimes {\overline s}_X$ and $w := \oneop_{\Xbar'} \otimes {\rbar_X^*} \cdot s_X \otimes \oneop_{\Xbar}$ are both in $\Hom_{\C^{(2)}}(\Xbar,\Xbar ')$ and they fulfill $w^*\cdot v = \oneop_{\Xbar}$, $v\cdot w^* = \oneop_{\Xbar '}$, thus the first statement is proven. The same 2-arrows $v$ and $w$, in the special case $\Xbar = \Xbar '$, fulfill the desired equalities in the second statement.
\end{proof}

\begin{remark}
For every two 1-arrows $X$ and $Y$, if there is an invertible 2-arrow $t:X\Rightarrow Y$ then there is also a unitary 2-arrow $u:X\Rightarrow Y$ defined by polar decomposition. The latter makes sense because we have a well defined continuous functional calculus on $\Hom_{\C^{(2)}}(X,X)$.
In particular, conjugates are uniquely determined up to unitary equivalence.
\end{remark}

When passing from $X$ to $\Xbar$ the roles of $\Z(\N)$ and $\Z(\M)$ are interchanged. By Frobenius reciprocity \cite[\Lem 2.1]{LoRo97}, the same is true for the left and right centers of $X$ and $\Xbar$.

\begin{lemma}\label{lem:lrcenters}
If $X$ and $\Xbar$ are conjugate 1-arrows in a 2-\Cstar-category, then $\Z^X_l(\N) \cong \Z^\Xbar_r(\N)$ and $\Z^X_r(\M) \cong \Z^\Xbar_l(\M)$ as unital commutative \Cstar-algebras and via the same isomorphism mapping $\Hom_{\C^{(2)}}(X,X)$ onto $\Hom_{\C^{(2)}}(\Xbar,\Xbar)$ (with the opposite multiplication). In particular, $\Z^X_l(\N) \cap \Z^X_r(\M) \cong \Z^\Xbar_l(\M) \cap \Z^\Xbar_r(\N)$.
\end{lemma}

We give the following

\begin{definition}
A 2-\Cstar-category $\C$ is said to have \textbf{finite-dimensional centers}, if $\Z(\N)=\Hom_{\C^{(2)}}(I_\N,I_\N)$ is finite-dimensional for every object $\N$ in $\C$, hence in particular $\Z(\N)\cong\CC^N$ as vector spaces, for some $N\in\NN$.
\end{definition}

Throughout this paper, we deal with \emph{rigid} 2-\Cstar-categories in order to have a finite notion of dimension (to be defined) for every 1-arrow. In addition, we shall assume \emph{finite-dimensionality} of the centers (in the sense described above), in order to translate the results of the previous sections and to keep our arguments (mainly decomposition theory and Perron-Frobenius theory) at a finite \Cstar-algebraic/categorical level. Under these assumptions the following fundamental structural result (which entails the \textbf{semisimplicity} of $\C$) holds. It generalizes \cite[\Lem 3.2]{LoRo97} from rigid tensor \Cstar-categories with simple tensor unit to rigid 2-\Cstar-categories with finite-dimensional centers. 

\begin{proposition}\label{prop:semisimp}
In a rigid 2-\Cstar-category $\C$ with finite-dimensional centers, $\Hom_{\C^{(2)}}(X,X)$ is a finite-dimensional \Cstar-algebra 
and $\Hom_{\C^{(2)}}(X,Y)$ is a finite-dimensional vector space for every $X,Y:\N\rightarrow\M$ 1-arrows in $\C$. 
If, in addition, $\C$ is closed under sub-1-arrows, then every 1-arrow is a finite direct sum of simple 1-arrows, \ie, $\C$ is semisimple.

More generally, let $X:\N\rightarrow\M$ be a 1-arrow admitting a conjugate 1-arrow $\Xbar:\M\rightarrow\N$ in a 2-\Cstar-category $\C$, then the following conditions are equivalent
\begin{itemize}
\item $\Hom_{\C^{(2)}}(X,X)$ is finite-dimensional.
\item The left center $\Z^X_l(\N) = \oneop_X \otimes \Hom_{\C^{(2)}}(I_\N,I_\N)$ is finite-dimensional.
\item The right center $\Z^X_r(\M) = \Hom_{\C^{(2)}}(I_\M,I_\M) \otimes \oneop_X$ is finite-dimensional.
\end{itemize}
\end{proposition}

\begin{remark}
In the special case of von Neumann algebra inclusions $\N\subset\M$ with finite Jones index, \ie, admitting an expectation in $E(\M,\N)$ with finite index, the previous proposition boils down to the already mentioned equivalence \cite[\Cor 3.18, 3.19]{BDH88} between the finite-dimensionality of $\N'\cap\M$, $\N'\cap\N$ or $\M'\cap\M$, as it will become clear in what follows.
\end{remark}

Before proving Proposition \ref{prop:semisimp}, we recall some terminology and results from \cite[\Sec 2]{LoRo97}, see also \cite[\Sec 1]{Zit07}. Let $X:\N\rightarrow\M$ and $\Xbar:\M\rightarrow\N$ be conjugate 1-arrows in a 2-\Cstar-category $\C$ (not necessarily with finite-dimensional centers), with $r_X$, $\overline r_X$ a solution of the conjugate equations. For every $t\in\Hom_{\C^{(2)}}(X,X)$, let
\begin{equation}\label{eq:lrinv}
t \mapsto r^*_X \cdot \oneop_{\Xbar} \otimes t \cdot r_X , \quad t \mapsto \rbar^*_X \cdot t \otimes \oneop_{\Xbar} \cdot \rbar_X
\end{equation}
be the \textbf{left inverse} and the \textbf{right inverse} of $X$ defined by $r_X$, $\overline r_X$, and denoted respectively by $t\mapsto\varphi_X(t)$ and $t\mapsto\psi_X(t)$. They map $\Hom_{\C^{(2)}}(X,X)$ respectively to $\Z(\N) = \Hom_{\C^{(2)}}(I_\N,I_\N)$ and to $\Z(\M) = \Hom_{\C^{(2)}}(I_\M,I_\M)$. 
Both these maps are linear, positive, mapping $\oneop_X$ respectively to $r^*_X r_X$ and to $\rbar^*_X \rbar_X$, and they are faithful due to Frobenius reciprocity \cite[\Lem 2.1]{LoRo97}. 
Now we consider the image of $\Z(\N)$ and $\Z(\M)$ in $\Z^X_l(\N)$ and $\Z^X_r(\M)$ respectively via $\otimes$-multiplication with $\oneop_X$ on the left and on the right, and we introduce the following \lqq symmetrically normalized" version of the left and right inverses of $X$ defined by $r_X$, $\overline r_X$, namely
$$t \mapsto (\rbar^*_X \cdot \rbar_X) \otimes \oneop_X \otimes (r^*_X \cdot \oneop_{\Xbar} \otimes t \cdot r_X), 
\quad t \mapsto (\rbar^*_X \cdot t \otimes \oneop_{\Xbar} \cdot \rbar_X) \otimes \oneop_X \otimes (r^*_X \cdot r_X)$$
denoted respectively by $t \mapsto \Phi_X(t)$ and $t \mapsto \Psi_X(t)$, and mapping $\Hom_{\C^{(2)}}(X,X)$ to its center.

\begin{lemma}\label{lem:pipobound}\emph{\cite[\Lem 2.7]{LoRo97}.}
In the previous notation, for every positive element $a^*\cdot a$ in $\Hom_{\C^{(2)}}(X,X)$ the following inequalities hold
$$a^*\cdot a \leq \Phi_X(a^* \cdot a),\quad a^*\cdot a \leq \Psi_X(a^* \cdot a)$$
as elements in the unital \Cstar-algebra $\Hom_{\C^{(2)}}(X,X)$. 

In particular, $\oneop_X \leq (\rbar^*_X \cdot \rbar_X) \otimes \oneop_X \otimes (r^*_X \cdot r_X)$ holds, thus $\oneop_X \otimes (r^*_X \cdot r_X)$ is invertible in $\Z^X_l(\N)$, bounded from below by $\|\rbar_X \otimes \oneop_X\|^{-2} \oneop_X$ and from above by $\| \oneop_X \otimes r_X\|^2 \oneop_X$. Similarly for $(\rbar^*_X \cdot \rbar_X) \otimes \oneop_X$ in $\Z^X_r(\M)$.
\end{lemma}

Hence, for every $t\in\Hom_{\C^{(2)}}(X,X)$
$$E_X(t) := ((\rbar^*_X \cdot \rbar_X) \otimes \oneop_X \otimes (r^*_X \cdot r_X))^{-1} \Phi_X(t),\quad E'_X(t) := ((\rbar^*_X \cdot \rbar_X) \otimes \oneop_X \otimes (r^*_X \cdot r_X))^{-1} \Psi_X(t)$$ 
define two faithful conditional expectations (projections of norm one between unital \Cstar-algebras) from $\Hom_{\C^{(2)}}(X,X)$ onto $\Z^X_l(\N)$ and $\Z^X_r(\M)$ respectively. Moreover, the inequalities in Lemma \ref{lem:pipobound} are actually optimal, see the comments after \cite[\Lem 2.7]{LoRo97}. In the case of trivial centers $\Z^X_l(\N) = \CC\oneop_X$ and $\Z^X_r(\M) = \CC\oneop_X$ they are equivalent to the \textbf{Pimsner-Popa bounds} \cite{PiPo86} for the expectations $E_X$ and $E'_X$ defined by the chosen solution $r_X,\rbar_X$ of the conjugate equations.

\begin{proof} (of Proposition \ref{prop:semisimp}).
We only have to show that if $\Z^X_l(\N)$ is finite-dimensional, where $X:\N\rightarrow\M$ is a 1-arrow with a conjugate $\Xbar:\M\rightarrow\N$, then $\Hom_{\C^{(2)}}(X,X)$ is also finite-dimensional.

Let $\oneop_X \otimes q_j$ be the minimal projections in $\Z^X_l(\N)$, for $j=1,\ldots,n$, where $n$ is the dimension of $\Z^X_l(\N)$ and $q_j\in\Z(\N)$. In particular, $\sum_j \oneop_X \otimes q_j = \oneop_X$. By considering the (finite) central decomposition of $\Hom_{\C^{(2)}}(X,X)$ given by the $\oneop_X \otimes q_j$, $j=1,\ldots,n$, and reducing the expectation $E_X$ accordingly, we may assume that $\Z^X_l(\N) = \CC\oneop_X$. In this case $\oneop_X \otimes (r^*_X \cdot r_X)$ is a positive number.
Now let $a_1,\ldots,a_N$ be arbitrary positive elements in $\Hom_{\C^{(2)}}(X,X)$ such that $\|a_k\| = 1$, $k=1,\ldots,N$, and $\sum_k a_k \leq \oneop_X$. By Lemma \ref{lem:pipobound}, for each $k=1,\ldots,N$ we have $a_k \leq (\rbar^*_X \cdot \rbar_X) \otimes \oneop_X \otimes (r^*_X \cdot r_X) E_X(a_k)$ hence $1 \leq \|(\rbar^*_X \cdot \rbar_X) \otimes \oneop_X\| \oneop_X \otimes (r^*_X \cdot r_X) E_X(a_k)$ because $E_X(a_k)$ is also a positive number. Summing over $k$ and using the positivity of $E_X$ we get $N \leq \|(\rbar^*_X \cdot \rbar_X) \otimes \oneop_X\| \oneop_X \otimes (r^*_X \cdot r_X)$, concluding the proof.
\end{proof}

Now, thanks to Proposition \ref{prop:semisimp}, for every 1-arrow $X:\N\rightarrow\M$ in a rigid 2-\Cstar-category with finite-dimensional centers $\C$, we can consider its \emph{finite} direct sum decomposition into \emph{simple} sub-1-arrows by considering the minimal projections in $\Hom_{\C^{(2)}}(X,X)$. 
We do it in different steps. We first consider the decomposition of $X:\N\rightarrow\M$ into connected sub-1-arrows, then the bi-central decomposition into factorial components, lastly the decomposition into simple ones.

Let $\{z_k, k=1,\ldots,l\}$ be the minimal projections in the algebra $\Z^X_l(\N) \cap \Z^X_r(\M) = (\oneop_X \otimes \Hom_{\C^{(2)}}(I_\N,I_\N)) \,\cap\, (\Hom_{\C^{(2)}}(I_\M,I_\M) \otimes \oneop_X)$, which we recall is a subalgebra of $\Z(\Hom_{\C^{(2)}}(X,X))$. Then the $z_k$ are unique up to permutation of the indices and $z_k = \oneop_X \otimes e_k = f_k \otimes \oneop_X$ for suitable elements
$e_k\in\Z(\N)$, $f_k\in\Z(\M)$. Now, for every $k=1,\ldots,l$, consider the sub-1-arrow of $X$ associated with $z_k$ and denoted by $X_k:\N\rightarrow\M$ with isometries $w_k\in\Hom_{\C^{(2)}}(X_k,X)$ such that $w_k \cdot w_k^* = z_k \leq \oneop_X$ and $w_k^* \cdot w_k = \oneop_{X_k}$.

\begin{lemma}\label{lem:connecteddecomp}
We have that $X = \oplus_k X_k$, where $X_k$ are connected sub-1-arrows of $X$.
\end{lemma}

\begin{proof}
The first statement follows immediately from the orthogonality of the $z_k$ for different $k$ and from $\sum_k z_k = \oneop_X$.

To prove the second statement, fix $k\in\{1,\ldots,l\}$ and let $a\in (\oneop_{X_k} \otimes \Hom_{\C^{(2)}}(I_\N,I_\N)) \,\cap\, (\Hom_{\C^{(2)}}(I_\M,I_\M) \otimes \oneop_{X_k})$, hence $a = \oneop_{X_k} \otimes b = c \otimes \oneop_{X_k}$ for some $b\in\Z(\N)$ and $c\in\Z(\M)$. 
But then $a = w_k^* \cdot w_k \cdot a \cdot w_k^* \cdot w_k$, regarding $a$ as an element of $\Hom_{\C^{(2)}}(X_k,X_k)$, and $w_k \cdot a \cdot w_k^*$ is an element of $\Hom_{\C^{(2)}}(X,X)$, or better of the subalgebra $(\oneop_{X} \otimes \Hom_{\C^{(2)}}(I_\N,I_\N)) \,\cap\, (\Hom_{\C^{(2)}}(I_\M,I_\M) \otimes \oneop_{X})$ because $w_k \cdot a \cdot w_k^* = \oneop_X \otimes (e_k \cdot b) = (c \cdot f_k) \otimes \oneop_X$. From its invariance under left and right multiplication with $z_k = w_k \cdot w_k^*$, which is minimal, we get $a\in\CC\oneop_{X_k}$.
\end{proof} 

\begin{remark}
Considering \emph{connected} 1-arrows as the building blocks for arbitrary 1-arrows, instead of factorial or simple ones, is the crucial step in order to define \lqq standard solutions" (see Definition \ref{def:stdsol}) of the conjugate equations for $X:\N\rightarrow\M$ in $\C$. Standard solutions generalize those defined in \cite{LoRo97} in the case of simple tensor units, and they correspond to the choice of the minimal conditional expectation in the special case of inclusions of von Neumann algebras $\N\subset\M$ with finite index and finite-dimensional centers considered in Section \ref{sec:inclusions}.
\end{remark}

Now let $X:\N\rightarrow\M$ be a connected 1-arrow in $\C$, and call $\{\oneop_X \otimes q_j,j=1,\ldots,n\}$, $\{p_i \otimes \oneop_X, i=1,\ldots,m\}$ the minimal projections in $\Z_l^X(\N)$ and $\Z_r^X(\M)$ respectively, where $n:=\dim(\Z_l^X(\N))$, $m:=\dim(\Z_r^X(\M))$. Notice that the left and right horizontal compositions with $\oneop_X$ might have a kernel, hence the $p_i$ and $q_j$ need not be projections themselves. Moreover, there might be more minimal projections in $\Z(\N)$ or $\Z(\M)$ that do not correspond to any of the previously fixed $n$ or $m$ projections and that are, in a sense, invisible for $X$.
Clearly the $\oneop_X \otimes q_j$ and $p_i \otimes \oneop_X$ are unique up to permutation of the indices.
Consider the product projections $p_i\otimes \oneop_X \otimes q_j$ in $\Z(\Hom_{\C^{(2)}}(X,X))$. 
If $p_i\otimes \oneop_X \otimes q_j \neq 0$, let $X_{ij}:\N\rightarrow\M$ be the associated sub-1-arrow of $X$ with isometry $v_{ij}\in\Hom_{\C^{(2)}}(X_{ij},X)$ such that $v_{ij}\cdot v_{ij}^* = p_i\otimes \oneop_X \otimes q_j \leq \oneop_X$ and $v_{ij}^* \cdot v_{ij} = \oneop_{X_{ij}}$. Otherwise, let $X_{ij}$ be a \emph{zero} 1-arrow in $\C$, namely such that $\Hom_{\C^{(2)}}(X_{ij}, X_{ij})$ is the zero vector space. In the latter case, notice that necessarily $\oneop_{X_{ij}} = 0$ and we may consider $X_{ij}$ as a sub-1-arrow of $X$ by choosing $v_{ij} = 0$ in $\Hom_{\C^{(2)}}(X_{ij},X)$.

\begin{lemma}\label{lem:factorialdecomp}
We have that $X = \oplus_{i,j} X_{ij}$, where $X_{ij}$ are either factorial sub-1-arrows of $X$, or zero 1-arrows.
\end{lemma}

\begin{proof}
As in the proof of the previous lemma, the first statement follows from the orthogonality of the $p_i\otimes \oneop_X \otimes q_j$ for different $i$ or $j$ and from $\sum_{i,j} p_i\otimes \oneop_X \otimes q_j = \oneop_X$.

To prove the second statement, fix $i\in\{1,\ldots,m\}$, $j\in\{1,\ldots,n\}$ such that $p_i\otimes \oneop_X \otimes q_j \neq 0$ and let $a$ be either in $\Z_l^{X_{ij}}(\N)$ or in $\Z_r^{X_{ij}}(\M)$, \ie, $a = \oneop_{X_{ij}} \otimes b$ or  $a = c \otimes \oneop_{X_{ij}}$ for some $b\in\Z(\N)$ or $c\in\Z(\M)$. Then $a = v_{ij}^* \cdot v_{ij} \cdot a \cdot v_{ij}^* \cdot v_{ij}$, where $v_{ij} \cdot a \cdot v_{ij}^*$ equals either $p_i \otimes \oneop_X \otimes (b \cdot q_j)$ or $ (c \cdot p_i) \otimes \oneop_X \otimes q_j$. From the minimality of $\oneop_X \otimes q_j$ or $p_i \otimes \oneop_X$ in the respective left/right center of $X$, we have that $v_{ij} \cdot a \cdot v_{ij}^*$ is a scalar multiple of $p_i \otimes \oneop_X \otimes q_j$, hence $a\in\CC\oneop_{X_{ij}}$.
\end{proof}

Finally, for every fixed $i\in\{1,\ldots,m\}$, $j\in\{1,\ldots,n\}$ such that $p_i\otimes \oneop_X \otimes q_j \neq 0$, consider the minimal projections $g_{ij,h}$ in $\Hom_{\C^{(2)}}(X_{ij},X_{ij})$ where $h$ runs in a finite set of indices, depending on $i$ and $j$. The corresponding sub-1-arrows $X_{ij,h}:\N\rightarrow\M$ of $X_{ij}$ with isometries $u_{ij,h}\in\Hom_{\C^{(2)}}(X_{ij,h},X_{ij})$ are simple, \ie, $\Hom_{\C^{(2)}}(X_{ij,h},X_{ij,h}) = \CC \oneop_{X_{ij,h}}$, and $X_{ij} = \oplus_h X_{ij,h}$, as one can immediately check. By Frobenius reciprocity \cite[\Lem 2.1]{LoRo97}, the 2-arrows $r_{ij,h}\in\Hom_{\C^{(2)}}(I_\N,\Xbar_{ij,h}\otimes X_{ij,h})$, $\rbar_{ij,h}\in\Hom_{\C^{(2)}}(I_\M,X_{ij,h}\otimes \Xbar_{ij,h})$ solving the conjugate equations for $X_{ij,h}$ and $\Xbar_{ij,h}$ are unique up to multiplication with scalars, compatibly with the equations themselves, once a conjugate $\Xbar_{ij,h}$ of $X_{ij,h}$ is chosen. 
In the same spirit of \cite[\Sec 3]{LoRo97} we may consider then \textbf{normalized solutions} of the conjugate equations, namely those fulfilling 
\footnote{Normalized solutions $r_X, \rbar_X$ of the conjugate equations, in the sense of \cite{LoRo97}, are those fulfilling $\|r_X\| = \|\rbar_X\|$. This condition is however equivalent to $\|\oneop_X \otimes r_X\| = \|\rbar_X \otimes \oneop_X\|$ in the case of 2-\Cstar-categories with simple tensor units, \ie, $\Z(\N) = \CC\oneop_{I_\N}$ for every object $\N$, because the $\otimes$-multiplication with $\oneop_X$ on the left and on the right considered in Definition \ref{def:simpleconnectedfactorial}, mapping $\Z(\N)$ onto $\Z_l^X(\N)$ and $\Z(\M)$ onto $\Z_r^X(\M)$, is trivially injective hence isometric in that case.}
$$\|\oneop_{X_{ij,h}} \otimes r_{ij,h}\| = \|\rbar_{ij,h} \otimes \oneop_{X_{ij,h}}\|.$$ 
The number
\begin{equation}\label{eq:sdsimple}
d_{X_{ij,h}} := \|\oneop_{X_{ij,h}} \otimes r_{ij,h}\|^2 = \|\oneop_{X_{ij,h}} \otimes (r_{ij,h}^* \cdot r_{ij,h})\|
\end{equation}
does not depend on the choice of normalized solution because the latter may vary only by multiplication with a complex phase, nor on the choice of conjugate.
We shall refer to $d_{X_{ij,h}}$ as the (scalar) dimension of $X_{ij,h}$ in this special case of simple 1-arrows.

\begin{lemma}\label{lem:normalizedsol}
Let $X$ and $\Xbar$ be conjugate 1-arrows in a 2-\Cstar-category with $r_X$, $\overline r_X$ a solution of the conjugate equations. Then the normalization condition $\|\oneop_X \otimes r_X\| = \|\rbar_X \otimes \oneop_X\|$ is equivalent to $\|\oneop_\Xbar \otimes \rbar_X\| = \|r_X \otimes \oneop_\Xbar\|$, and in this case the four norms are all equal.

If $X$, or equivalently $\Xbar$, is a factorial 1-arrow, then the normalization condition reads
$$\oneop_X \otimes (r^*_X \cdot r_X) = (\rbar^*_X \cdot \rbar_X) \otimes \oneop_X,$$
or equivalently
$$\oneop_\Xbar \otimes (\rbar^*_X \cdot \rbar_X) = (r^*_X \cdot r_X) \otimes \oneop_\Xbar,$$
as scalar multiples of $\oneop_X$ in $\Hom_{\C^{(2)}}(X,X)$, or of $\oneop_\Xbar$ in $\Hom_{\C^{(2)}}(\Xbar,\Xbar)$, respectively, and in this case the two scalars are equal.  
\end{lemma}

\begin{proof}
It is enough to observe that, by the \Cstar identity on the norms and by Lemma \ref{lem:lrcenters}, 
$\|\oneop_X \otimes r_X\| = \|r_X \otimes \oneop_\Xbar\|$ and $\|\oneop_\Xbar \otimes \rbar_X\| = \|\rbar_X \otimes \oneop_X\|$ always hold. The second statement is then a consequence of $\|\oneop_X\| = \|\oneop_\Xbar\| = 1$.
\end{proof}

By Frobenius reciprocity $\Xbar_{ij,h}$ is a simple sub-1-arrow of $\Xbar_{ij}$ for any choice of conjugate 1-arrow of $X_{ij}$, and by the previous lemma the choice of normalized solutions $r_{ij,h}$, $\rbar_{ij,h}$ is symmetric in $X_{ij,h}$ and $\Xbar_{ij,h}$, moreover $d_{X_{ij,h}} = d_{\Xbar_{ij,h}} \geq 1$. 

Now, let ${\overline u}_{ij,h}\in\Hom_{\C^{(2)}}(\Xbar_{ij,h}, \Xbar_{ij})$ be a family of isometries expressing the direct sum decomposition $\Xbar_{ij} = \oplus_h \Xbar_{ij,h}$, independently of the previously fixed family of isometries $u_{ij,h}$ expressing the decomposition $X_{ij} = \oplus_h X_{ij,h}$. Similarly to \cite[\Sec 3]{LoRo97}, we define
$$r_{ij} := \sum_h {\overline u}_{ij,h} \otimes u_{ij,h} \cdot r_{ij,h}, \quad \rbar_{ij} := \sum_h {u}_{ij,h} \otimes {\overline u}_{ij,h} \cdot \rbar_{ij,h}$$
also written as
$$r_{ij} := \oplus_h r_{ij,h}, \quad \rbar_{ij} := \oplus_h \rbar_{ij,h}.$$
The 2-arrows $r_{ij}\in\Hom_{\C^{(2)}}(I_\N,\Xbar_{ij}\otimes X_{ij})$, $\rbar_{ij}\in\Hom_{\C^{(2)}}(I_\M,X_{ij}\otimes \Xbar_{ij})$ form a solution of the conjugate equations for $X_{ij}$ and $\Xbar_{ij}$, as one can directly check. Moreover, a solution defined in this way is again \emph{normalized}, indeed 
\begin{equation}\label{eq:dijadditive}
\|\oneop_{X_{ij}} \otimes (r_{ij}^* \cdot r_{ij})\| = \|\oneop_{X_{ij}} \otimes (\sum_h r_{ij,h}^* \cdot r_{ij,h})\|
= \sum_h \|\oneop_{X_{ij,h}} \otimes (r_{ij,h}^* \cdot r_{ij,h})\|
\end{equation}
because $\oneop_{X_{ij}} \otimes (r_{ij,h}^* \cdot r_{ij,h})$ and $\oneop_{X_{ij,h}} \otimes (r_{ij,h}^* \cdot r_{ij,h})$ are both positive scalar multiples respectively of $\oneop_{X_{ij}}$ and $\oneop_{X_{ij,h}}$ by the factoriality of $X_{ij}$ and $X_{ij,h}$ (the latter are also simple), and they are the \emph{same} scalar by Lemma \ref{lem:whiskerisrep}, moreover $\|\oneop_{X_{ij}}\| = \|\oneop_{X_{ij,h}}\| = 1$ for every $h$.
Thus $\|\oneop_{X_{ij}} \otimes (r_{ij}^* \cdot r_{ij})\| = \|(\rbar_{ij}^* \cdot \rbar_{ij}) \otimes \oneop_{X_{ij}}\|$ because each $r_{ij,h}, \rbar_{ij,h}$ is normalized. The number
\begin{equation}\label{eq:sdfactorial}
d_{X_{ij}} := \|\oneop_{X_{ij}} \otimes r_{ij}\|^2 = \|\oneop_{X_{ij}} \otimes (r_{ij}^* \cdot r_{ij})\|
\end{equation}
does not depend on the choice of solution $r_{ij}, \rbar_{ij}$ constructed in this way (because the normalized solutions $r_{ij,h}, \rbar_{ij,h}$ for each $h$ may vary only up to a phase, and because the choice of orthonormal bases of isometries ${\overline u}_{ij,h}$ and $u_{ij,h}$ is clearly irrelevant when computing $r_{ij}^* \cdot r_{ij}$), nor on the choice of conjugate. 
We call this number the (scalar) dimension of the factorial 1-arrow $X_{ij}$, and we call any solution $r_{ij}, \rbar_{ij}$ constructed as above a \textbf{standard solution} of the conjugate equations for $X_{ij}$ and $\Xbar_{ij}$.

By \cite[\Lem 3.3, 3.7]{LoRo97}, with obvious modifications from the case of simple tensor units to the case of factorial 1-arrows considered here, we have the following strengthening of Lemma \ref{lem:uniqueconjandsol}. 

\begin{lemma}\label{lem:uniquestdij}
For any two standard solutions $r_{ij}, \rbar_{ij}$ and $r'_{ij}, \rbar'_{ij}$ of the conjugate equations for $X_{ij}$ and $\Xbar_{ij}$, there is a unitary 2-arrow $u$ in $\Hom_{\C^{(2)}}(\Xbar_{ij},\Xbar_{ij})$ such that $r'_{ij} = u \otimes \oneop_{X_{ij}} \cdot r_{ij}$ and $\rbar'_{ij} = \oneop_{X_{ij}} \otimes u \cdot \rbar_{ij}$. Similarly with a unitary 2-arrow in $\Hom_{\C^{(2)}}(X_{ij},X_{ij})$.
\end{lemma}
 
Notably, by the computation shown in equation (\ref{eq:dijadditive}), the scalar dimension of factorial 1-arrows (\cf \cite[\p 114]{LoRo97}) is \emph{additive}, namely
$$d_{X_{ij}} = \sum_h d_{X_{ij,h}}$$
where $d_{X_{ij,h}}$ is defined in equation (\ref{eq:sdsimple}). Moreover, by Lemma \ref{lem:normalizedsol} we have $d_{X_{ij}} = d_{\Xbar_{ij}} \geq 1$.

\begin{remark}
Notice however that the factoriality of $X_{ij}$ is crucial in the proof of additivity of $d_{X_{ij}}$. Indeed there is no reason to expect that the number $\| \oneop_X \otimes (r^*_X\cdot r_X)\|$, defined by some solution of the conjugate equations $r_X, \rbar_X$ for a \emph{non-factorial} 1-arrow $X$, is equal to $\sum_h \| \oneop_{X_h} \otimes (r^*_{X_h}\cdot r_{X_h})\|$, where $X = \oplus_h X_h$, not even if the $X_h$ are simple sub-1-arrows of $X$ and if we consider as before $r_X:= \oplus_h r_{X_h}, \rbar_X := \oplus_h \rbar_{X_h}$.
\end{remark}

We introduce below another notion of standardness for solutions of the conjugate equations, for \emph{arbitrary} 1-arrows, which is motivated by the theory of minimal index in the von Neumann algebra context and which boils down to the one given in \cite{LoRo97} in the factorial case. With this notion we obtain a \emph{weighted} additivity formula for the scalar dimension, generalizing the one valid for the square root of the minimal index presented in Theorem \ref{thm:minimalconnectedfindim}, equation (\ref{eq:weightedadditivd}). We show then that standardness is intrinsically characterized by \lqq sphericality" and, independently, \lqq minimality" properties of the solutions themselves.

In what follows, let $\C$ be a rigid 2-\Cstar-category with finite-dimensional centers.

\begin{definition}
Let $X$ be a connected 1-arrow in $\C$, and denote as before $n=\dim(\Z_l^X(\N))$, $m=\dim(\Z_r^X(\M))$. We define the \textbf{matrix dimension}, or simply \textbf{dimension}, of $X$ as the $m\times n$ matrix $D_X$ with entries
$$(D_X)_{i,j} := d_{X_{ij}} \quad \text{or}\quad (D_X)_{i,j} := 0$$
for $i=1,\ldots,m$, $j= 1,\ldots,n$, depending on whether $X_{ij}$ is a non-zero factorial sub-1-arrow appearing in the decomposition of $X$, hence $d_{X_{ij}}$ is the number defined in equation (\ref{eq:sdfactorial}), or a zero 1-arrow (see Lemma \ref{lem:factorialdecomp}). The \textbf{scalar dimension} of $X$ is the number 
$$d_X := \|D_X\|$$ 
where $\|\cdot\|$ denotes the $l^2$-operator norm.

For an arbitrary 1-arrow $X$ in $\C$, we consider first the decomposition into connected sub-1-arrows $X_k$, $k=1\ldots,l$ (see Lemma \ref{lem:connecteddecomp}), and we define the matrix dimension and the scalar dimension of $X$ respectively as
$$D_X := \oplus_k D_{X_k}, \quad d_X := \|D_X\| = \max_{k}\, d_{X_k}$$
where $\oplus$ denotes the direct sum of matrices. 

If we consider the vector of scalar dimensions of the connected components $\vec{d}_X := (d_{X_1}, \ldots, d_{X_l})^t$, that we may call the \textbf{vector dimension} of $X$, then $d_X = \|\vec{d}_X\|_\infty$ by definition, where $\|\cdot\|_\infty$ is the $l^\infty$-norm on vectors.
\end{definition}

\begin{remark}\label{rmk:Dconj}
Notice that the scalar dimension $d_X$ coincides with the one defined in equations (\ref{eq:sdsimple}) and (\ref{eq:sdfactorial}) for simple and factorial 1-arrows respectively. Moreover, $d_X \geq 1$ by normalization, and 
$$D_{\Xbar} = (D_X)^t$$ 
by Lemma \ref{lem:lrcenters}. In particular, $d_X = d_{\Xbar}$. The matrix dimension is invariant under unitary equivalence, namely $D_X = D_Y$ if there is a unitary 2-arrow in $\Hom_{\C^{(2)}}(X,Y)$.
\end{remark}

Considering connected 1-arrows as the building blocks for this theory of dimension is due to the following fact, which generalizes Lemma \ref{lem:connectediffirred}. 

\begin{lemma}
Let $X$ be an arbitrary 1-arrow in $\C$ and consider its decomposition into factorial or zero sub-1-arrows $X=\oplus_{i,j} X_{ij}$ as in Lemma \ref{lem:factorialdecomp} (but without the connectedness assumption). Let $S$ be the $m\times n$ matrix with the $(i,j)$-th entry equal to $1$ (or any positive number) if $X_{ij}$ is non-zero and equal to $0$ if $X_{ij}$ is zero. Then $X$ is connected if and only if $S$ is indecomposable, \ie, if and only if $S S^t$, or equivalently $S^t S$, are irreducible square matrices.
\end{lemma}

In particular, the matrix dimension $D_X$ of a connected 1-arrow $X$ is indecomposable, thus, using Theorem \ref{thm:minimalconnectedfindim}, equations (\ref{eq:PF-eqns}) and (\ref{eq:DdeterminesLambda}) as a guideline, we consider the \emph{unique} positive $l^2$-normalized Perron-Frobenius eigenvectors $\nu_X^{1/2}$ and $\mu_X^{1/2}$ respectively of $(D_X)^t D_X$ and $D_X (D_X)^t$ associated with the eigenvalue $d^{\,2}_X$ (the common spectral radius), namely
\begin{align*}
(D_X)^t D_X\, \nu_X^{1/2} &= d_X^2\, \nu_X^{1/2}\\
D_X (D_X)^t \mu_X^{1/2} &= d_X^2\, \mu_X^{1/2}.
\end{align*}

\begin{remark}\label{rmk:weightedadditivdcat}
Notice that $\nu_{\Xbar}^{1/2} = \mu_X^{1/2}$ and $\mu_{\Xbar}^{1/2} = \nu_X^{1/2}$, and that, \eg, $\nu_X^{1/2} = 1$ if $(D_X)^t D_X$ is a number (namely if it equals $d^{\,2}_X$), \ie, if $n=1$. 
Moreover, the previous Perron-Frobenius eigenvalue equations are equivalent to $D_X \nu_X^{1/2} = d_X \mu_X^{1/2}$, $(D_X)^t \mu_X^{1/2} = d_X \nu_X^{1/2}$, hence from $\|\nu_X^{1/2}\| = \|\mu_X^{1/2}\| = 1$ we obtain 
\begin{equation}\label{eq:disDangles}
d_X = \langle D_X \nu_X^{1/2}| \mu_X^{1/2}\rangle = \langle \nu_X^{1/2}| (D_X )^t \mu_X^{1/2}\rangle
\end{equation}
where $\langle\cdot|\cdot\rangle$ denotes the $l^2$-scalar product, \cf equation (\ref{eq:weightedadditivd}) in Theorem \ref{thm:minimalconnectedfindim}.
\end{remark}

In the notation of Lemma \ref{lem:connecteddecomp}, \ref{lem:factorialdecomp}, see also Lemma \ref{lem:lrcenters}, we give the following

\begin{definition}\label{def:stdsol}
Let $X$ and $\Xbar$ be conjugate 1-arrows in $\C$, and assume that $X$, or equivalently $\Xbar$, is connected. For every $i=1,\ldots,m$, $j= 1,\ldots,n$, if $X_{ij}$ is a non-zero factorial sub-1-arrow of $X$, choose a standard solution $r_{ij}$, $\rbar_{ij}$ of the conjugate equations for $X_{ij}$ and $\Xbar_{ij}$, set $r_{ij} = \rbar_{ij} = 0$ otherwise. We define
$$r_X := \sum_{i,j} \frac{\mu_{X,i}^{1/4}}{\nu_{X,j}^{1/4}}\, \overline{v}_{ij} \otimes v_{ij} \cdot r_{ij}, \quad 
\rbar_X := \sum_{i,j} \frac{\nu_{X,j}^{1/4}}{\mu_{X,i}^{1/4}}\, v_{ij} \otimes \overline{v}_{ij} \cdot \rbar_{ij}$$
also written as
$$r_X = \oplus_{i,j} \frac{\mu_{X,i}^{1/4}}{\nu_{X,j}^{1/4}}\, r_{ij}, \quad \rbar_X = \oplus_{i,j} \frac{\nu_{X,j}^{1/4}}{\mu_{X,i}^{1/4}}\, \rbar_{ij}$$
where $\nu_{X,j}^{1/4}$, $\mu_{X,i}^{1/4}$ denote the square roots of the entries of the eigenvectors $\nu_X^{1/2}$, $\mu_X^{1/2}$, and $v_{ij}\in\Hom_{\C^{(2)}}(X_{ij}, X)$, $\overline{v}_{ij}\in\Hom_{\C^{(2)}}(\Xbar_{ij}, \Xbar)$ are families of isometries with ranges $p_i \otimes \oneop_X \otimes q_j$ and $q_j \otimes \oneop_\Xbar \otimes p_i$ realizing the decompositions $X = \oplus_{i,j} X_{ij}$ and $\Xbar = \oplus_{i,j} \Xbar_{ij}$.

For an arbitrary pair of conjugate 1-arrows $X$ and $\Xbar$ in $\C$, consider first the decompositions into connected sub-1-arrows $X = \oplus_k X_k$ and $\Xbar = \oplus_k \Xbar_k$, and set
$$r_X := \oplus_k r_{X_k}, \quad \rbar_X := \oplus_k \rbar_{X_k}$$
where $\oplus$ are defined as before by means of families of isometries with ranges $z_k = f_k \otimes \oneop_X \otimes e_k$ and $\tilde z_k = e_k \otimes \oneop_\Xbar \otimes f_k$ realizing the decompositions. 

We call the 2-arrows $r_X$, $\rbar_X$ defined in this way a \textbf{standard solution} of the conjugate equations for $X$ and $\Xbar$ in $\C$.
\end{definition}

Let's first check that the \lqq weights" $\mu_{X,i}^{1/4}\nu_{X,j}^{-1/4}$ and $\nu_{X,j}^{1/4}\mu_{X,i}^{-1/4}$ that we add on the direct summands defining $r_X$ and $\rbar_X$ do not break the conjugate equations. Indeed
\begin{align*}
\rbar^*_X \otimes \oneop_X \cdot \oneop_X \otimes r_X 
&= \sum_{i,j} \frac{\nu_{X,j}^{1/4}}{\mu_{X,i}^{1/4}}\, \rbar^*_{ij} \otimes \oneop_X \cdot v^*_{ij} \otimes \overline{v}^*_{ij} \otimes \oneop_X \cdot \sum_{i',j'} \frac{\mu_{X,i'}^{1/4}}{\nu_{X,j'}^{1/4}}\, \oneop_X \otimes \overline{v}_{i'j'} \otimes v_{i'j'} \cdot \oneop_X \otimes r_{i'j'}\\
&= \sum_{i,j} \rbar^*_{ij} \otimes \oneop_X \cdot v^*_{ij} \otimes \oneop_{\Xbar_{ij}} \otimes v_{ij} \cdot \oneop_X \otimes r_{ij}\\
&= \sum_{i,j} v_{ij} \cdot \rbar^*_{ij} \otimes \oneop_{X_{ij}} \cdot \oneop_{X_{ij}} \otimes r_{ij} \cdot v^*_{ij} = \sum_{i,j} v_{ij} \cdot v^*_{ij} = \oneop_X
\end{align*}
using the fact that our weights are not affected by complex conjugation (they are positive real numbers), using orthogonality of the $\overline{v}_{ij}$, completeness of the $v_{ij}$, one of the conjugate equations for each $r_{ij}$, $\rbar_{ij}$, the interchange law and $\oneop_{I_\M} \otimes v_{ij} = v_{ij}$, $v_{ij} \otimes \oneop_{I_\N}= v_{ij}$. Similarly $r^*_X \otimes \oneop_\Xbar \cdot \oneop_\Xbar \otimes \rbar_X = \oneop_\Xbar$, hence $r_X$, $\rbar_X$ are indeed a solution of the conjugate equations as we claimed.  

\begin{proposition}\label{prop:stdisnormalized}
Let $X$ and $\Xbar$ be conjugate 1-arrows in $\C$. Then any standard solution $r_X$, $\rbar_X$ of the conjugate equations for $X$ and $\Xbar$ is normalized, \ie, we have $\|\oneop_X \otimes (r_X^* \cdot r_X)\| = \|(\rbar_X^* \cdot \rbar_X) \otimes \oneop_X\| = d_X$. Moreover, if $X$, or equivalently $\Xbar$, is connected, then
\begin{equation}\label{eq:stdisscalarifconnected}
\oneop_X \otimes (r_X^* \cdot r_X) = (\rbar_X^* \cdot \rbar_X) \otimes \oneop_X = d_X \,\oneop_X
\end{equation}
\ie, both 2-arrows are scalar multiples of $\oneop_X$ in $\Hom_{\C^{(2)}}(X,X)$, and the scalar is $d_X$, \footnote{This is a generalization of the fact that $\Ind(E^0)$ is a \emph{scalar} operator in $\Z(\M)$, namely $\Ind(E^0) = c \oneop$, where $c = \|\Ind(E^0)\|$ is the minimal index of a connected inclusion of von Neumann algebras $\N\subset\M$ as in Section \ref{sec:inclusions} and $E^0$ is the minimal expectation in $E(\M,\N)$.}.

If $X$ and $\Xbar$ are arbitrary, then 
\begin{equation}\label{eq:stdisscalaroplus}
\oneop_X \otimes (r_X^* \cdot r_X) = (\rbar_X^* \cdot \rbar_X) \otimes \oneop_X = \oplus_k d_{X_k} \oneop_{X_k}
\end{equation}
where $X=\oplus_k X_k$ is the decomposition into connected components as in Lemma \ref{lem:connecteddecomp}.

Similarly exchanging $X$ with $\Xbar$ and $r_X$ with $\rbar_X$, and recalling that $d_X = d_{\Xbar}$.
In particular, neither $\oneop_X \otimes (r_X^* \cdot r_X)$, nor $(\rbar_X^* \cdot \rbar_X) \otimes \oneop_X$, nor their norm depend on the choice of standard solutions.
\end{proposition}

\begin{proof}
Assume first that $X$, hence $\Xbar$, is connected and compute
$$\oneop_X \otimes (r_X^* \cdot r_X) = \sum_{i,j} \frac{\mu_{X,i}^{1/2}}{\nu_{X,j}^{1/2}}\, \oneop_X \otimes (r_{ij}^* \cdot r_{ij}) = \sum_{i,j,j'} \frac{\mu_{X,i}^{1/2}}{\nu_{X,j}^{1/2}}\, (t_{j'} \cdot \oneop_{X_{j'}} \cdot t_{j'}^*) \otimes (r_{ij}^* \cdot \oneop_{\Xbar_{ij}} \otimes \oneop_{X_{ij}} \cdot r_{ij})$$
where $X_{j'}$ is the sub-1-arrow of $X$ given by a projection $\oneop_X \otimes q_{j'}$ in $\Hom_{\C^{(2)}}(X,X)$, minimal in $\Z_l^X(\N)$ if $X:\N\rightarrow\M$, and $t_{j'}\in\Hom_{\C^{(2)}}(X_{j'},X)$ is the associated isometry such that $t_{j'}\cdot t^*_{j'} = \oneop_X \otimes q_{j'}$. Notice that $\sum_{j'} \oneop_X \otimes q_{j'} = \oneop_{X}$. Moreover, $\oneop_{X_{j'}} = t^*_{j'}\cdot \oneop_X \otimes q_{j'} \cdot t_{j'} = \oneop_{X_{j'}} \otimes q_{j'}$ for every $j' = 1,\ldots,n$, and similarly we get $\oneop_{X_{ij}} = p_i \otimes \oneop_{X_{ij}} \otimes q_j$ and $\oneop_{\Xbar_{ij}} = q_j \otimes \oneop_{\Xbar_{ij}} \otimes p_i$ for every $i=1,\ldots,m$, $j=1,\ldots,n$, even when $X_{ij}$ and $\Xbar_{ij}$ are zero 1-arrows. In particular, $\oneop_{X_{j'}} \otimes \oneop_{\Xbar_{ij}} = \delta_{j,j'} \oneop_{X_j} \otimes \oneop_{\Xbar_{ij}}$ because, \eg, $\oneop_X \otimes (q_{j'}\cdot q_j) = \oneop_X \otimes q_{j'} \cdot \oneop_X \otimes q_j = 0$ and $\oneop_{X_{j'}} \otimes (\cdot)$ is a subrepresentation of $\oneop_X \otimes (\cdot)$ for $\Z(\N)$ by Lemma \ref{lem:whiskerisrep}. Thus we continue
$$=  \sum_{i,j} \frac{\mu_{X,i}^{1/2}}{\nu_{X,j}^{1/2}}\, t_{j} \cdot \oneop_{X_{j}} \otimes (r_{ij}^* \cdot r_{ij}) \cdot t_{j}^* 
= \sum_{i,j} \frac{\mu_{X,i}^{1/2}}{\nu_{X,j}^{1/2}}\, (D_X)_{i,j}\, t_{j} \cdot t_{j}^*
= d_X\, \oneop_X$$
using the fact that $\oneop_{X_{j}} \otimes (r_{ij}^* \cdot r_{ij})$ is a scalar multiple of $\oneop_{X_{j}}$ because the left center of $X_{j}$ is trivial \footnote{We might have introduced the notion of \lqq left or right factoriality" for 1-arrows, \cf Definition \ref{def:simpleconnectedfactorial}, but we refrain from doing so at the moment.}, this scalar is $d_{X_{ij}}$ (or $0$ if $X_{ij}$ is zero) because $\oneop_{X_{ij}} \otimes (r_{ij}^* \cdot r_{ij}) = d_{X_{ij}} \oneop_{X_{ij}}$ by Lemma \ref{lem:normalizedsol} and equation (\ref{eq:sdfactorial}), and we invoke again Lemma \ref{lem:whiskerisrep} on the sub-1-arrow $X_{ij}$ of $X_j$ given by the projection $p_i \otimes \oneop_{X_j}$. 
The last equality is obtained by summing first over $i$ and using $(D_X)^t \mu_X^{1/2} = d_X \nu_X^{1/2}$ (see Remark \ref{rmk:weightedadditivdcat}), thus each $\nu_{X,j}^{1/2}$ cancels, and then summing over $j$. The equality $(\rbar_X^* \cdot \rbar_X) \otimes \oneop_X = d_X\,\oneop_X$ is obtained from a symmetric argument using the other equation $D_X \nu_X^{1/2} = d_X \mu_X^{1/2}$, thus we have shown equation (\ref{eq:stdisscalarifconnected}) and the normalization condition follows. 

In the general case, equation (\ref{eq:stdisscalaroplus}) and the normalization condition follow from the very definition of standard solutions for non-connected 1-arrows, indeed $\oneop_X \otimes (r_X^* \cdot r_X) = (\rbar_X^* \cdot \rbar_X) \otimes \oneop_X = \sum_k d_{X_k} z_k$, where $z_k$ are the minimal projections in $\Z^X_l(\N) \cap \Z^X_r(\M)$. 
\end{proof}

\begin{remark}\label{rmk:non-std}
If we drop the weights given by the Perron-Frobenius eigenvectors $\nu_X^{1/2}$ and $\mu_X^{1/2}$ in the definition of standard solutions for a connected 1-arrow $X$, and consider $\oplus_{i,j} r_{ij}$, $\oplus_{i,j} \rbar_{ij}$ instead, where the $r_{ij}$, $\rbar_{ij}$ are standard and normalized, what we get is again a solution of the conjugate equations (by the same computation), but this solution is in general \emph{not} normalized, nor it fulfills the analogue of equation (\ref{eq:stdisscalarifconnected}) in Proposition \ref{prop:stdisnormalized}. The normalization condition is satisfied if and only if the 1-arrow $X$ itself fulfills $\|D_X\|_\infty = \|D_X\|_1$ (see \cite[\Sec 5.6.4, 5.6.5]{HoJoBook} for the definition of the $l^\infty$ and $l^1$ matrix norms), while the analogue of equation (\ref{eq:stdisscalarifconnected}) holds if and only if $d_X^{-1}D_X$ is both column and row-stochastic, hence in particular if $\|D_X\|_\infty = \|D_X\|_1 = \|D_X\|_2$. The condition on $d_X^{-1}D_X$ is equivalent to $D_X 1_n = d_X 1_m$, $(D_X)^t 1_m = d_X 1_n$, where $1_k$ denotes the vector in $\CC^k$ with coordinates all equal to $1$, thus we get $n=m$ and $\nu_{X,j}^{1/2} = \mu_{X,i}^{1/2} = 1/\sqrt{n}$ for every $i$ and $j$. In other words, a solution $\oplus_{i,j} r_{ij}$, $\oplus_{i,j} \rbar_{ij}$ fulfills the analogue of equation (\ref{eq:stdisscalarifconnected}), as any standard solution does, if and only if $\oplus_{i,j} r_{ij}$, $\oplus_{i,j} \rbar_{ij}$ is itself standard, and this can only happen if the dimensions of the left and right center of $X$ are equal.

Moreover, $\oplus_{i,j} r_{ij}$, $\oplus_{i,j} \rbar_{ij}$ does \emph{not} correspond in general to the minimal conditional expectation in the special case of connected von Neumann algebra inclusions, \cf equation (\ref{eq:DdeterminesLambda}) in Theorem \ref{thm:minimalconnectedfindim}. We shall come back later to this minimality property, and generalize it to 2-\Cstar-categories.  
\end{remark}

Standard solutions $r_X$, $\overline r_X$, or better the associated left and right inverses $\varphi_X$ and $\psi_X$ defined on the unital \Cstar-algebra $\Hom_{\C^{(2)}}(X,X)$ by equation (\ref{eq:lrinv}), with values in $\Z(\N)$ and $\Z(\M)$ respectively, enjoy the following \emph{trace property}, which does \emph{not} however characterize them among all the solutions of the conjugate equations. In order to show the trace property we need a preliminary lemma.

\begin{lemma}\label{lem:inequivijk}
Let $X:\N\rightarrow\M$ be a 1-arrow in $\C$ (not necessarily connected) and let $X = \oplus_{i,j} X_{ij}$ be the decomposition into factorial or zero sub-1-arrows as in Lemma \ref{lem:factorialdecomp}. Then $\Hom_{\C^{(2)}}(X_{ij},X_{i'j'})$ is the zero vector space whenever $i\neq i'$ or $j\neq j'$.

Similarly, let $X = \oplus_k X_k$ be the decomposition into connected sub-1-arrows as in Lemma \ref{lem:connecteddecomp}. Then $\Hom_{\C^{(2)}}(X_k,X_{k'})$ is the zero vector space whenever $k \neq k'$.
\end{lemma}

\begin{proof}
Simply observe that $t \in \Hom_{\C^{(2)}}(X_{ij},X_{i'j'})$ fulfills $t = \oneop_{X_{i'j'}} \cdot t \cdot \oneop_{X_{ij}} = p_{i'} \otimes \oneop_{X_{i'j'}} \otimes q_{j'} \cdot \oneop_{I_\M} \otimes t \otimes \oneop_{I_\N} \cdot p_i \otimes \oneop_{X_{ij}} \otimes q_j = (p_i \cdot p_{i'}) \otimes t \otimes (q_j \cdot q_{j'}) = \delta_{i,i'} \delta_{j,j'} t$ by the orthogonality of the $\oneop_X \otimes q_j$ and of the $p_i \otimes \oneop_X$, and by Lemma \ref{lem:whiskerisrep}, as in the proof of Proposition \ref{prop:stdisnormalized}.

Similarly for $\Hom_{\C^{(2)}}(X_k,X_{k'})$ by observing that $\oneop_{X_k} = \oneop_{X_k} \otimes e_k = f_k \otimes \oneop_{X_k}$ and by the orthogonality of the $\oneop_X \otimes e_k = f_k \otimes \oneop_X$.
\end{proof}

\begin{proposition}\label{prop:trace}
Let $X$ and $\Xbar$ be conjugate 1-arrows in $\C$. Let $r_X$, $\rbar_X$ be a standard solution of the conjugate equations for $X$ and $\Xbar$. Then the associated left and right inverses of $X$ are tracial, namely
$$\varphi_X(s \cdot t) = \varphi_X(t \cdot s), \quad \psi_X(s \cdot t) = \psi_X(t \cdot s)$$
for every $s,t\in \Hom_{\C^{(2)}}(X,X)$. Either equality is equivalent to the following one 
\begin{equation}\label{eq:bullet}
\oneop_\Xbar \otimes \rbar_X^* \cdot \oneop_\Xbar \otimes t \otimes \oneop_\Xbar \cdot r_X \otimes \oneop_\Xbar = r_X^* \otimes \oneop_\Xbar \cdot \oneop_\Xbar \otimes t \otimes \oneop_\Xbar \cdot \oneop_\Xbar \otimes \rbar_X
\end{equation}
between elements of $\Hom_{\C^{(2)}}(\Xbar,\Xbar)$, for every $t\in \Hom_{\C^{(2)}}(X,X)$.
\end{proposition}

\begin{proof}
The stated equivalence between either trace property and equation (\ref{eq:bullet}) follows by the conjugate equations and Frobenius reciprocity, with the same arguments of \cite[\Lem 2.3 (c)]{LoRo97}, \cite[\Prop 2.6]{BKLR15}. See the latter reference for a proof using graphical calculus and observe that the triviality assumption on the centers plays no role there.

Hence we have to show that standardness of $r_X$, $\rbar_X$ does imply equation (\ref{eq:bullet}). Assume first that $X$, hence $\Xbar$, is connected. By Lemma \ref{lem:inequivijk} applied to $\Xbar$, we only have to show the equality after multiplying on the left with $\overline{v}_{ij}^*$ and on the right with $\overline{v}_{ij}$, for every $i,j$, where $\overline{v}_{ij}\in\Hom_{\C^{(2)}}(\Xbar_{ij}, \Xbar)$ is the family of isometries in the definition of $r_X$, $\rbar_X$. The coefficients coming from $\mu_{X,i}^{1/4}$ and $\nu_{X,j}^{1/4}$ in the definition of $r_X$, $\rbar_X$ cancel on both sides, and we know by \cite[\Lem 3.7]{LoRo97}, \cite[\Prop 2.4, 2.6]{BKLR15} that standardness of each $r_{ij}$, $\rbar_{ij}$ implies the analogue of equation (\ref{eq:bullet}) for factorial 1-arrows $X_{ij}$, $\Xbar_{ij}$, with $t$ replaced by $v_{ij}^*\cdot t \cdot v_{ij}$. 
Now, dropping the connectedness assumption on $X$, again by Lemma \ref{lem:inequivijk} applied to $\Xbar$, we can check equation (\ref{eq:bullet}) after multiplying on the left with $\overline{w}_k^*$ and on the right with $\overline{w}_k$, for every $k$, where $\overline{w}_{k}\in\Hom_{\C^{(2)}}(\Xbar_{k}, \Xbar)$ is the family of isometries in the definition of $r_X$, $\rbar_X$. By the previous step applied to each $w_k^*\cdot t \cdot w_k$ we have the statement.
\end{proof}

\begin{remark}
If $t\in\Hom_{\C^{(2)}}(X,X)$, then equation (\ref{eq:bullet}) applied to $t^*$ is the definition of the antilinear involutive multiplication-preserving \lqq bullet" map $t\mapsto t^{\bullet}\in\Hom_{\C^{(2)}}(\Xbar,\Xbar)$ considered in \cite[\Sec 2]{LoRo97}, such that $t^{*\bullet} = t^{\bullet *}$, also called conjugate map in \cite[\Sec 2]{HePe15}, or better conjugate functor, and denoted by $t\mapsto \overline{t}$. Note that $\oneop_X^\bullet = \oneop_\Xbar$ is a reformulation of the conjugate equations.

If $t\in\Hom_{\C^{(2)}}(X,Y)$, where $X,Y$ are both connected and non-factorial 1-arrows, the analogue of equation (\ref{eq:bullet}) does \emph{not} hold in general for standard solutions $r_X$, $\rbar_X$ and $r_Y$, $\rbar_Y$, because the matrix dimensions $D_X$ and $D_Y$ may be different and the contributions from $\mu_{X,i}^{1/4}$, $\mu_{Y,i'}^{1/4}$ and $\nu_{X,j}^{1/4}$, $\nu_{X,j'}^{1/4}$ need no longer cancel.

On another hand, one can consider \emph{non-standard} solutions of the conjugate equations for conjugate 1-arrows $X$ and $\Xbar$ (not necessarily factorial, nor connected), \eg, direct sums of standard solutions over the factorial components $\oplus_{i,j} r_{ij}$, $\oplus_{i,j} \rbar_{ij}$ (\cf Remark \ref{rmk:non-std}), or direct sums of normalized solutions over the simple components $\oplus_{h} r_{h}$, $\oplus_{h} \rbar_{h}$ (choosing a decomposition $X = \oplus_h X_h$ into simple sub-1-arrows, \ie, choosing a decomposition of $\oneop_X$ by minimal projections in $\Hom_{\C^{(2)}}(X,X)$).
Both these options fulfill equation (\ref{eq:bullet}), \ie, they give rise to traces on the unital \Cstar-algebra $\Hom_{\C^{(2)}}(X,X)$, by arguments similar to those presented in Proposition \ref{prop:trace} and \cite[\Lem 3.7]{LoRo97}, see also the proof of \cite[\Thm 4.22]{BDH14}. 
Using these non-standard solutions one can define non-standard maps $t\mapsto t^\bullet$ in $\C$, even for 1-arrows $t\in\Hom_{\C^{(2)}}(X,Y)$ with $X\neq Y$. Alternatively, one can embed $\Hom_{\C^{(2)}}(X,Y)$ as a corner of $\Hom_{\C^{(2)}}(X \oplus Y, X \oplus Y)$ and use again standard solutions for the direct sum 1-arrow to define $t\mapsto t^\bullet$ in $\C$.
\end{remark}

The trace property guarantees the \emph{uniqueness} of standard solutions (up to unitaries), generalizing Lemma \ref{lem:uniquestdij} from factorial to arbitrary 1-arrows $X$ in $\C$.

\begin{lemma}\label{lem:uniquestd}
For any two standard solutions $r_X, \rbar_X$ and $r'_X, \rbar'_X$ of the conjugate equations for $X$ and $\Xbar$, there is a unitary 2-arrow $u$ in $\Hom_{\C^{(2)}}(\Xbar,\Xbar)$ such that $r'_X = u \otimes \oneop_{X} \cdot r_X$ and $\rbar'_X = \oneop_{X} \otimes u \cdot \rbar_X$. Similarly with a unitary 2-arrow in $\Hom_{\C^{(2)}}(X,X)$.
\end{lemma}

\begin{proof}
We first state a slight generalization of equation (\ref{eq:bullet}), namely
$$\oneop_\Xbar \otimes \rbar_X^* \cdot \oneop_\Xbar \otimes t \otimes \oneop_\Xbar \cdot r'_X \otimes \oneop_\Xbar = r_X^* \otimes \oneop_\Xbar \cdot \oneop_\Xbar \otimes t \otimes \oneop_\Xbar \cdot \oneop_\Xbar \otimes \rbar'_X$$ 
for every $t\in\Hom_{\C^{(2)}}(X,X)$, whenever $r_X, \rbar_X$ and $r'_X, \rbar'_X$ are standard. The proof is analogous to the one contained in Lemma \ref{prop:trace}.

Now, setting $t=\oneop_X$ and using the conjugate equations for both standard solutions, we have that the 2-arrow $u := \oneop_{\Xbar} \otimes \rbar_X^* \cdot r'_X \otimes \oneop_{\Xbar}$ in $\Hom_{\C^{(2)}}(\Xbar,\Xbar)$ is unitary and relates the two standard solutions as desired.
\end{proof}

We introduce now two \emph{canonical} states $\omega_l^X$ and $\omega_r^X$ respectively on the left and right center of a connected 1-arrow $X$ in $\C$, namely two further invariants extracted from the matrix dimension $D_X$. We characterize then standardness of a solution of the conjugate equations for $X$ and $\Xbar$ by means of an intrinsic property called sphericality. 

\begin{definition}\label{def:lrstates}
Let $X:\N\rightarrow\M$ be a connected 1-arrow in $\C$, let $D_X$ be its matrix dimension and let $\nu_X^{1/2}$, $\mu_X^{1/2}$ be the associated unique Perron-Frobenius eigenvectors considered before. We define two faithful states $\omega_l^X: \Z^X_l(\N) \rightarrow \CC$ and $\omega_r^X: \Z^X_r(\M) \rightarrow \CC$ by setting their values on minimal projections as
$$\omega_l^X(\oneop_X \otimes q_j) := \nu_{X,j}, \quad \omega_r^X(p_i \otimes \oneop_X) := \mu_{X,i}$$
for every $j=1,\ldots,n$, $i=1,\ldots,m$, where $\nu_{X,j}$, $\mu_{X,i}$ denote the squares of the entries of $\nu_X^{1/2}$, $\mu_X^{1/2}$, thus $\omega_l^X(\oneop_X) = 1$, $\omega_r^X(\oneop_X) = 1$. We refer to them as the \textbf{left} and \textbf{right state} of $X$.
\end{definition}

\begin{remark}
Observe that passing from $X$ to a conjugate $\Xbar$ simply interchanges the roles of left and right, \ie, $\omega_l^\Xbar = \omega_r^X$ and $\omega_r^\Xbar = \omega_l^X$.
\end{remark}

\begin{remark}\label{rmk:non-preferredlrstates}
If $X$ is not connected, we do not have a preferred choice of left and right state of $X$. One can consider the decomposition $X = \oplus_k X_k$ into connected components as in Lemma \ref{lem:connecteddecomp}, and define
$$\omega_l^X(\cdot) := \oplus_k \alpha_k\, \omega_l^{X_k}(\cdot), \quad \omega_r^X(\cdot) := \oplus_k \beta_k\, \omega_r^{X_k}(\cdot)$$
on $\Z^X_l(\N)$ and $\Z^X_r(\M)$ respectively, where $\alpha_k$, $\beta_k$ are arbitrary numbers. This definition makes sense because $\Z^X_l(\N)$ is the image of $\Z(\N)$ under a direct sum representation $\oneop_X \otimes (\cdot) = \oplus_k (\oneop_{X_k} \otimes (\cdot))$ by Lemma \ref{lem:whiskerisrep}, and similarly for $\Z^X_r(\M)$.
Moreover, we shall always assume to have $\alpha_k = \beta_k\geq 0$ for every $k$, so to have two positive linear functionals. These functionals are respectively faithful and normalized, if $\alpha_k > 0$ for every $k$ and $\sum_k \alpha_k = 1$.
\end{remark}

We can now state and prove the previously announced equivalence between standardness and sphericality for solutions of the conjugate equations.

\begin{theorem}\label{thm:stdiffspherical}
Let $X$ and $\Xbar$ be conjugate 1-arrows in $\C$, and assume that they are connected. Let $r_X$, $\rbar_X$ be a standard solution of the conjugate equations for $X$ and $\Xbar$. Then the two positive linear functionals on $\Hom_{\C^{(2)}}(X,X)$ obtained by composing the associated left/right inverses $\varphi_X$, $\psi_X$ with the canonical left/right states $\omega_l^X$, $\omega_r^X$ of $X$ coincide, namely
\begin{equation}\label{eq:sphericality}
\omega_l^X(\oneop_X \otimes \varphi_X(t)) = \omega_r^X(\psi_X(t) \otimes \oneop_X)
\end{equation}
for every $t\in \Hom_{\C^{(2)}}(X,X)$. We call this property \textbf{sphericality} of the solution $r_X$, $\rbar_X$.

Vice versa, every solution fulfilling sphericality is necessarily standard. 
\end{theorem}

Equation (\ref{eq:sphericality}) is independent of the choice of standard solutions by Lemma \ref{lem:uniquestd} and, after dividing by $\oneop_X \otimes (r_X^* \cdot r_X) = (\rbar_X^* \cdot \rbar_X) \otimes \oneop_X = d_X \oneop_X$, it is a generalization of the spherical state $\omega_s$ of a connected inclusion of von Neumann algebras considered in Section \ref{sec:inclusions}. Recall that $E_X = d_X^{-1} \oneop_X \otimes \varphi_X(\cdot)$ and $E'_X = d_X^{-1} \psi_X(\cdot) \otimes \oneop_X$ are conditional expectations from $\Hom_{\C^{(2)}}(X,X)$ onto $\Z_l^X(\N)$ and $\Z_r^X(\M)$ respectively.

\begin{definition}\label{def:sstate}
Let $X$ be a connected 1-arrow in $\C$, then
$$\omega_s^X(\cdot) := \omega_l^X(E_X(\cdot)) = \omega_r^X(E'_X(\cdot))$$
is a faithful state $\omega_s^X: \Hom_{\C^{(2)}}(X,X) \rightarrow \CC$. We call it the \textbf{spherical state} of $X$.
\end{definition}

One can check (\ref{eq:sphericality}) on $t = p_i \otimes \oneop_X \otimes q_j$ by computing
\begin{equation}\label{eq:sphericalityonij}
\oneop_X \otimes \varphi_X(t) = \frac{\mu_{X,i}^{1/2}}{\nu_{X,j}^{1/2}} (D_X)_{i,j}\, \oneop_X \otimes q_j, \quad
\psi_X(t) \otimes \oneop_X = \frac{\nu_{X,j}^{1/2}}{\mu_{X,i}^{1/2}} (D_X)_{i,j}\, p_i \otimes \oneop_X,
\end{equation}
hence we get the number $\mu_{X,i}^{1/2} \nu_{X,j}^{1/2} (D_X)_{i,j}$ on either side, by applying $\omega_l^X$ or $\omega_r^X$. Summing over $i,j$ and using equation (\ref{eq:stdisscalarifconnected}) we re-obtain the identity (\ref{eq:disDangles}) for the scalar dimension $d_X$.

\begin{proof} (of Theorem \ref{thm:stdiffspherical}).
Assume that $r_X$, $\rbar_X$ is standard, then the proof of equation (\ref{eq:sphericality}) relies on the same computation leading to (\ref{eq:sphericalityonij}), with $(D_X)_{i,j}$ replaced by the number $t_{ij}$ such that $\oneop_{X_{ij}} \otimes \varphi_{X_{ij}}(v_{ij}^*\cdot t \cdot v_{ij}) = t_{ij}\, \oneop_{X_{ij}}$, $v_{ij}\in\Hom_{\C^{(2)}}(X_{ij},X)$, and on the sphericality property of standard solutions in the factorial case \cite[\Lem 3.9]{LoRo97}, \cite[\Prop 2.5]{BKLR15}, namely $\oneop_{X_{ij}} \otimes \varphi_{X_{ij}}(s) = \psi_{X_{ij}}(s) \otimes \oneop_{X_{ij}}$ for every $s\in\Hom_{\C^{(2)}}(X_{ij},X_{ij})$. 

Conversely, a solution fulfilling equation (\ref{eq:sphericality}) is necessarily standard by the same argument used in \cite[\Lem 3.9]{LoRo97}, because we have Lemma \ref{lem:uniqueconjandsol}, we have shown the trace property for standard solutions in Proposition \ref{prop:trace}, and the left/right inverses and states are faithful.
\end{proof}

\begin{remark}\label{rmk:sphericalitynon-connected}
If $X$ is not connected, by taking direct sums one can show that equation (\ref{eq:sphericality}) still holds for standard solutions, for any choice of functionals $\omega_l^X$, $\omega_r^X$ as in Remark \ref{rmk:non-preferredlrstates}. Conversely, equation (\ref{eq:sphericality}) characterizes standardness if we choose $\omega_l^X$, $\omega_r^X$ to be faithful, \eg $\alpha_k = 1$ for every $k$, or $\alpha_k = 1/p$ where $p = \dim(\Z^X_l(\N) \cap \Z^X_r(\M))$.
\end{remark}

\begin{remark}
Similar in spirit, but inequivalent notions of sphericality (beyond the factorial or simple tensor units case) have been considered in the literature of bimodules \cite[\Sec 4]{BDH14} and 2-\Cstar-categories \cite[\Def 3.5, 3.12]{Zit07}, and used to single out \lqq special" solutions of the conjugate equations. The first is close to the unweighted direct sum solution $\oplus_{i,j} r_{ij}$, $\oplus_{i,j} \rbar_{ij}$ discussed in Remark \ref{rmk:non-std}, 
the second is given for \lqq centrally balanced" 1-arrows, a notion which is quite orthogonal to connectedness, that we regard as the most natural to develop and generalize the theory of minimal index.
\end{remark}

Using the spherical state, we give below a generalization of the trace formula contained in Proposition \ref{prop:trace} to the case $X \neq Y$.

\begin{proposition}
Let $X,Y:\N\rightarrow\M$ be connected 1-arrows in $\C$. Assume that $\oneop_Z\otimes(\cdot)$ and $(\cdot)\otimes\oneop_Z$, for $Z=X,Y$, are isomorphisms of $\Z(\N)$ and $\Z(\M)$ onto the respective left/right center, and denote as before by $p_i \otimes \oneop_X$, $p_i \otimes \oneop_Y$, $\oneop_X \otimes q_j$, $\oneop_Y \otimes q_j$ the minimal projections in the left/right centers, where $i=1,\ldots,m$, $j=1,\ldots,n$ and $n=\dim(\Z_l^X(\N))=\dim(\Z_l^Y(\N))$, $m=\dim(\Z_r^X(\M))=\dim(\Z_r^Y(\M))$. Then 
$$\frac{\omega_s^X(p_i \otimes (t^*\cdot t) \otimes q_j)}{\omega_s^Y(p_i \otimes (t\cdot t^*) \otimes q_j)} = \frac{d_Y}{d_X} \frac{\mu_{X,i}^{1/2} \nu_{X,j}^{1/2}}{\mu_{Y,i}^{1/2} \nu_{Y,j}^{1/2}}$$
for every $t\in\Hom_{\C^{(2)}}(X,Y)$.

If $X,Y$ are factorial 1-arrows, then
$$\frac{\omega_s^X(t^*\cdot t)}{\omega_s^Y(t\cdot t^*)} = \frac{d_Y}{d_X}.$$
\end{proposition}

\begin{proof}
The first statement is obtained by direct computation, using the definition of standard solutions for $X,Y$, and the trace property in the factorial case $\varphi_{X_{ij}}(t_{ij}^* \cdot t_{ij}) = \varphi_{Y_{ij}}(t_{ij} \cdot t_{ij}^*)$, for every $t_{ij} \in \Hom_{\C^{(2)}}(X_{ij},Y_{ij})$, see \cite[\Lem 3.7]{LoRo97}, \cite[\Prop 2.4]{BKLR15}. 
The second statement follows because $D_X = d_X$, $D_Y = d_Y$ and $\mu_X^{1/2} = \nu_X^{1/2} = 1$, $\mu_Y^{1/2} = \nu_Y^{1/2} = 1$ in the factorial case. 
\end{proof}

If $X$ and $\Xbar$ are connected, standard solutions are \textbf{minimal} among all other solutions, in the sense that they minimize the number $\|\oneop_X \otimes (r_X^*\cdot r_X)\| \, \|(\rbar_X^*\cdot \rbar_X) \otimes \oneop_X\|$, and by Proposition \ref{prop:stdisnormalized} the minimum is precisely the square of the scalar dimension of $X$ (that we may call the \textbf{minimal index} of $X$, or of $\Xbar$, in $\C$). Moreover, this property is another intrinsic characterization of standardness.

\begin{theorem}\label{thm:stdiffminimal}
Let $X$ and $\Xbar$ be conjugate 1-arrows in $\C$, and assume that they are connected. 
Then any solution $r_X$, $\rbar_X$ of the conjugate equations for $X$ and $\Xbar$ fulfills
\begin{equation}\label{eq:minimality}
\|\oneop_X \otimes (r_X^*\cdot r_X)\| \, \|(\rbar_X^*\cdot \rbar_X) \otimes \oneop_X\| \geq d^{\,2}_X
\end{equation}
where the equality is attained if and only if the solution is standard.
\end{theorem}

\begin{proof}
Consider the canonical left and right states $\omega_l^X$ and $\omega_r^X$ of $X$, and observe that 
$$\|\oneop_X \otimes (r_X^*\cdot r_X)\| \geq \omega_l^X(\oneop_X \otimes (r_X^*\cdot r_X)),\quad \|(\rbar_X^*\cdot \rbar_X) \otimes \oneop_X\| \geq \omega_r^X((\rbar_X^*\cdot \rbar_X) \otimes \oneop_X)$$
because $\omega_l^X$, $\omega_r^X$ have norm 1. By Lemma \ref{lem:uniqueconjandsol} there is an invertible 2-arrow $w$ in $\Hom_{\C^{(2)}}(X,X)$ such that $r_X = \oneop_\Xbar \otimes w \cdot s_X$ and $\rbar_X = {w^*}^{-1} \otimes \oneop_\Xbar \cdot {\overline s}_X$ and $s_X$, ${\overline s}_X$ is standard. Denote by $\varphi_X$, $\psi_X$ the left/right inverses of $X$ associated with $s_X$, ${\overline s}_X$. Thus $\omega_l^X(\oneop_X \otimes (r_X^*\cdot r_X)) = \omega_l^X(\oneop_X \otimes \varphi_X(w^*\cdot w))$ and $\omega_r^X((\rbar_X^*\cdot \rbar_X) \otimes \oneop_X) = \omega_r^X(\psi_X(w^{-1} \cdot {w^*}^{-1}) \otimes \oneop_X) = \omega_l^X(\oneop_X \otimes \varphi_X(w^{-1} \cdot {w^*}^{-1}))$ by sphericality of standard solution (Theorem \ref{thm:stdiffspherical}). Now, as in the proof of \cite[\Thm 3.11]{LoRo97} we take the spectral decomposition of $w^*\cdot w = \sum_\alpha \mu_\alpha e_\alpha$ in $\Hom_{\C^{(2)}}(X,X)$, where $\mu_\alpha$ are distinct eigenvalues, thus $(w^*\cdot w)^{-1} = \sum_\alpha \mu_\alpha^{-1} e_\alpha$, and
\begin{align*}
\omega_l^X(\oneop_X \otimes (r_X^*\cdot r_X))\, \omega_r^X((\rbar_X^*\cdot \rbar_X) \otimes \oneop_X) 
&= \sum_{\alpha, \alpha'} \mu_\alpha \mu_{\alpha'}^{-1} \omega_l^X(\oneop_X \otimes \varphi_X(e_\alpha)) \, \omega_l^X(\oneop_X \otimes \varphi_X(e_{\alpha'})) \\
&\geq \omega_l^X(\oneop_X \otimes \varphi_X(\oneop_X))^2 = d^{\,2}_X
\end{align*}
by the same estimate $\mu_\alpha \mu_{\alpha'}^{-1} + \mu_{\alpha'} \mu_\alpha^{-1} \geq 2$ used in \cite[\Thm 3.11]{LoRo97} and by Proposition \ref{prop:stdisnormalized} (in the connected case) applied to $s_X$, ${\overline s}_X$ in the last equality. The last inequality is indeed an equality if and only if $1$ is the only eigenvalue of $w^*\cdot w$, \ie, if and only if $w$ is unitary and $r_X$, $\rbar_X$ is itself standard. Thus by Proposition \ref{prop:stdisnormalized} we have the second statement.
\end{proof}

\begin{remark}
If $X$ and $\Xbar$ are not connected, let $X = \oplus_k X_k$, $\Xbar = \oplus_k \Xbar_k$ be the decompositions into connected components as in Lemma \ref{lem:connecteddecomp}. The inequality (\ref{eq:minimality}) is again fulfilled by any solution $r_X$, $r_\Xbar$, indeed the same proof works (see Remark \ref{rmk:sphericalitynon-connected}) for arbitrary states $w_l^X$, $w_r^X$ defined as in Remark \ref{rmk:non-preferredlrstates}. We can take, \eg, $\alpha_k = 1/q$ if $k$ corresponds to one of the maximal entries $d_{X_k}$ of the vector dimension $\vec{d}_X$, with $q$ the number of times that this maximal value (equal to the scalar dimension $d_X$) occurs, and $\alpha_k = 0$ otherwise, in order to have $\omega_l^X(\oneop_X \otimes \varphi_X(\oneop_X)) = d_X$. The equality in (\ref{eq:minimality}) does not characterize standardness anymore, as the state considered here is not necessarily faithful and the two members of (\ref{eq:minimality}) only give information on the connected components of $X$ with maximal scalar dimension.

In general, given a solution $r_X$, $\rbar_X$ one can write 
$$\oneop_X \otimes (r_X^* \cdot r_X) = \oplus _k (\oneop_{X_k} \otimes (r_{X_k}^* \cdot r_{X_k})), \quad 
(\rbar_X^* \cdot \rbar_X) \otimes \oneop_X = \oplus _k ((\rbar_{X_k}^* \cdot \rbar_{X_k}) \otimes \oneop_{X_k})$$
where $r_{X_k} := \overline{w}_k^* \otimes w_k^* \cdot r_X$, $\rbar_{X_k} := w_k^* \otimes \overline{w}_k^* \cdot \rbar_X$, and $w_k\in\Hom_{\C^{(2)}}(X_k,X)$, $\overline{w}_k\in\Hom_{\C^{(2)}}(\Xbar_k,\Xbar)$ are isometries with ranges $z_k = f_k \otimes \oneop_X \otimes e_k$ and $\tilde{z}_k = e_k \otimes \oneop_\Xbar \otimes f_k$ realizing $X = \oplus_k X_k$, $\Xbar = \oplus_k \Xbar_k$. Then $r_X$, $\rbar_X$ is standard if and only if each $r_{X_k}$, $\rbar_{X_k}$ is standard (by definition), \ie, if the latter fulfill (\ref{eq:minimality}) with the equality sign.
\end{remark}

As already checked (with different techniques) in the special case of von Neumann algebra inclusions, see Section \ref{sec:multiplicativity} and \ref{sec:additivity}, we conclude by showing \textbf{additivity} and \textbf{multiplicativity} of the matrix dimension in this more general 2-\Cstar-categorical setting.

\begin{proposition}
Let $X:\N\rightarrow\M$ be a 1-arrow in $\C$ with matrix dimension $D_X$, and assume (without loss of generality) that it is connected. Let $f_\alpha$, $\alpha=1,\ldots,N$, be mutually orthogonal projections in $\Hom_{\C^{(2)}}(X,X)$ such that $\sum_\alpha f_\alpha = \oneop_X$. Then 
$$D_X = \sum_\alpha D_{X_\alpha}$$ 
where $X_\alpha$ is the sub-1-arrow of $X$ corresponding to $f_\alpha$ and $D_{X_\alpha}$ is the matrix dimension defined using the minimal projections in $\Z_l^X(\N)$ and $\Z_r^X(\M)$.
\end{proposition}

\begin{proof}
Let $\oneop_X \otimes q_j$, $p_i \otimes \oneop_X$ be the minimal projections in the left and right center of $X$ respectively, and let $X = \oplus_{i,j} X_{ij}$ be the decomposition into factorial or zero sub-1-arrows, realized by isometries $v_{ij}$ with range $p_i\otimes \oneop_X \otimes q_j$, as in Lemma \ref{lem:factorialdecomp}. Then for fixed $i,j$, the 2-arrows $v_{ij}^*\cdot f_\alpha \cdot v_{ij}\in\Hom_{\C^{(2)}}(X_{ij}, X_{ij})$, $\alpha=1,\ldots,N$, are mutually orthogonal projections (possibly zero), because $f_\alpha$ and $p_i \otimes \oneop_X \otimes q_j$ commute, and they sum up to $\oneop_{X_{ij}}$, thus $X_{ij} = \oplus_{\alpha} X_{ij,\alpha}$ where $X_{ij,\alpha}$ are the associated sub-1-arrows of $X_{ij}$. 
By decomposing further each $v_{ij}^*\cdot f_\alpha \cdot v_{ij}$ into minimal projections in $\Hom_{\C^{(2)}}(X_{ij}, X_{ij})$, and by the computation shown in (\ref{eq:dijadditive}) in the factorial case (or trivially if $X_{ij} = 0$), we obtain $d_{X_{ij}} = \sum_\alpha d_{X_{ij,\alpha}}$. Moreover, $d_{X_{ij,\alpha}} = d_{X_{\alpha,ij}}$ because $X_{ij,\alpha}$ and $X_{\alpha,ij}$ are unitarily equivalent, thus the statement.
\end{proof}

\begin{proposition}
Let $X:\N\rightarrow\M$ and $Y:\M\rightarrow\L$ be 1-arrows in $\C$ with matrix dimensions $D_X$, $D_Y$ respectively, and assume (without loss of generality) that they are connected. Assume in addition that $\Z_r^X(\M)$ and $\Z_l^Y(\M)$ are isomorphic and that the isomorphism commutes with the two representations $(\cdot) \otimes \oneop_X$ and $\oneop_Y \otimes (\cdot)$ of $\Z(\M)$ (this happens, in particular, if the latter are also isomorphisms). Then 
$$D_{Y\otimes X} = D_Y D_X$$ 
where $D_{Y\otimes X}$ is the matrix dimension defined using the minimal projections in $\Z_l^X(\N)$ and $\Z_r^Y(\L)$.
\end{proposition}

\begin{proof}
Let $\oneop_X \otimes q_j$, $r_k \otimes \oneop_Y$ be the minimal projections in $\Z_l^X(\N)$ and $\Z_r^Y(\L)$ respectively. By assumption, the minimal projections in $\Z_r^X(\M)$ and $\Z_l^Y(\M)$ can be written as $p_i \otimes \oneop_X$ and $\oneop_Y \otimes p_i$ with respect to the same $p_i\in\Z(\M)$, and $\oneop_Y \otimes (p_i \otimes \oneop_X) = (\oneop_Y \otimes p_i) \otimes \oneop_X$ by strict associativity of the horizontal composition in $\C$. Let $Y \otimes X = \oplus_{k,j} (Y \otimes X)_{kj}$ be the decomposition into factorial or zero sub-1-arrows, realized by isometries $v_{kj}$ with range $r_k\otimes \oneop_{Y \otimes X} \otimes q_j$, analogue to Lemma \ref{lem:factorialdecomp}. With the obvious notation, we have that $(Y\otimes X)_{kj}$ and $Y_k \otimes X_j$ are unitarily equivalent. Now, for fixed $k,j$, the 2-arrows $v_{kj}^*\cdot \oneop_Y \otimes p_i \otimes \oneop_X \cdot v_{kj}\in\Hom_{\C^{(2)}}((Y \otimes X)_{kj}, (Y \otimes X)_{kj})$ are mutually orthogonal projections (possibly zero), because $\oneop_Y \otimes p_i \otimes \oneop_X$ and $r_k\otimes \oneop_{Y \otimes X} \otimes q_j$ commute, and they sum up to $\oneop_{(Y \otimes X)_{kj}}$, thus $(Y \otimes X)_{kj} = \oplus_i (Y \otimes X)_{kj,i}$ and $d_{(Y \otimes X)_{kj}} = \sum_i d_{(Y \otimes X)_{kj,i}}$ as in equation (\ref{eq:dijadditive}). Observing that $(Y \otimes X)_{kj,i}$ is unitarily equivalent to $Y_{ki} \otimes X_{ij}$, again in the obvious notation, and using the multiplicativity of the scalar dimension in the factorial case \cite[\Cor 3.10]{LoRo97}, we get $d_{(Y \otimes X)_{kj,i}} = d_{Y_{ki}} d_{X_{ij}}$ hence the proof is complete.
\end{proof}

As a consequence of these last two propositions, one can prove statements analogue to those contained in Section \ref{sec:multiplicativity} and \ref{sec:additivity} on the (sub)multiplicativity of the scalar dimension $d_X$ in the case of connected (non-factorial) 1-arrows $X$, on the sufficient conditions for sharp multiplicativity, on the additivity of $d_X^{\, 2}$ (\ie, additivity of the minimal index of $X$) when either the left or right center of $X$ is trivial and we decompose along the remaining right or left central projections, together with the weighted additivity formula for $d_X$ obtained in (\ref{eq:disDangles}), \cf Theorem \ref{thm:minimalconnectedfindim}, equation (\ref{eq:weightedadditivd}).

In particular, in view of the multiplicativity property of the minimal index in the Jones tower case (Corollary \ref{cor:jonesextensionmultip}), let $Y$ be a self-conjugate 1-arrow of the form $Y=\Xbar\otimes X:\N\rightarrow\N$ and assume to have a braiding $\eps_{Y,Y}\in\Hom_{\C^{(2)}}(Y\otimes Y,Y\otimes Y)$ arising in some specific example of rigid 2-\Cstar-category with finite-dimensional centers $\C$. A natural question is to exploit the relationship between the Markov trace arising from the spherical state $\omega_s^{Y^n}$ of $Y\otimes\cdots\otimes Y$ obtained by iterating the standard left inverse of $Y$ (Definition \ref{def:sstate}) and knot invariants as in the Jones polynomial case \cite{Jon87}.





\bigskip
{\bf Acknowledgements.}
We thank the Isaac Newton Institute (INI) for Mathematical Sciences in Cambridge for hospitality during the program \lqq Operator algebras: subfactors and their applications" (OAS), supported by EPSRC grant numbers EP/K032208/1 and EP/R014604/1.
L.G. thanks also the Department of Mathematics of the Ohio University and of the Georg-August-Universit\"at G\"ottingen for financial support, the latter during the workshop LQP41 \lqq Foundations and constructive aspects of QFT" where part of our results were presented.
We acknowledge the MIUR Excellence Department Project awarded to the Department of Mathematics, University of Rome Tor Vergata, CUP E83C18000100006.

\small


\def\cprime{$'$}

\end{document}